\documentclass{amsart}

\usepackage{amsfonts}
\usepackage{amsmath}
\usepackage{amssymb}
\usepackage{amsthm}

\makeatletter
\renewcommand\@seccntformat[1]{\csname the#1\endcsname.\quad}

\newtheorem{theorem}{Theorem}[section]

\newtheorem{corollary}[theorem]{Corollary}

\newtheorem{example}[theorem]{Example}

\newtheorem{lemma}[theorem]{Lemma}

\newtheorem{proposition}[theorem]{Proposition}
\newtheorem{remark}[theorem]{Remark}

\newcommand{\norm}[3]{\ensuremath{\left\Vert#1\right\Vert_{#2}^{#3}}}
\newcommand{\abs}[3]{\ensuremath{\left\vert#1\right\vert_{#2}^{#3}}}
\DeclareMathOperator{\supp}{supp}
\DeclareMathOperator{\dist}{dist}
\DeclareMathOperator{\sg}{sgn}

\DeclareMathOperator{\Card}{Card}

\begin{document}

\author[M. Ciesielski]{Maciej Ciesielski }

\title[Relationships between $K$-monotonicity and rotundity properties]{Relationships between $K$-monotonicity and rotundity properties with application}

\begin{abstract}
	In this paper we investigate a relationship between fully $k$-rotundity properties, uniform $K$-monotonicity properties, reflexivity and $K$-order continuity in a symmetric spaces $E$. We also answer a crucial question whether fully $k$-rotundity properties might be restricted in definition to $E^d$ the positive cone of all nonnegative and decreasing elements of $E$. 
	We present a complete characterization of decreasing uniform $K$-monotonicity and $K$-order continuity in $E$. 
	It is worth mentioning that we also establish several auxiliary results describing reflexivity in Lorentz spaces $\Gamma_{p,w}$ and $K$-order continuity in Orlicz spaces $L^\psi$. Finally, we show an application of discussed geometric properties to the approximation theory.
	
\end{abstract}

\maketitle

\bigskip\ 

{\small \underline{2000 Mathematics Subjects Classification: 46E30, 46B20, 46B28.}\hspace{1.5cm}\
\ \ \quad\ \quad  }\smallskip\ 

{\small \underline{Key Words and Phrases:}\hspace{0.15in} Decreasing (increasing) uniform $K$-monotonicity, $K$-order continuity, strict $K$-monotonicity, fully $k$-rotundity, symmetric space, the best approximation.}

\bigskip\ \ 

\section{Introduction}

In 1955, K. Fan and I. Glicksberg \cite{FaGl} introduced and characterized fully $k$-rotundity properties in Banach spaces. Recently, in the papers \cite{CHK,HuKoLe}, there have been shown, among others, a complete correspondence between fully $k$-rotundity properties, rotundity and reflexivity with an application to the approximation theory in Banach spaces. The next interesting results were published in \cite{CerHudzKamMas}, where authors have investigated, inter alia, rotundity properties on $E^d$ the positive cone of all nonnegative and decreasing elements of a $K$-monotone symmetric space $E$. The further  motivation of our investigation can be found in \cite{Cies-K-mono,Cies-SKM&KOC,Cies-JAT,CieLewUKM}, where authors have presented a correspondence and complete criteria for $K$-order continuity, strict $K$-monotonicity and uniform $K$-monotonicity properties with application to the best dominated approximation problems in the sense of
the Hardy-Littlewood-P\'olya relation. The main idea of this paper is to find a relationship between reflexivity, fully $k$-rotundity properties and uniform $K$-monotonicity properties with application to the approximation theory. In view of the previous research, we also focus on full criteria for $K$-order continuity and fully $k$-rotundity properties in symmetric spaces.

The article is organized as follows. Section 2 contains all the necessary definitions and notation, which are used in our discussion. 
In Section 3, we show a relationship between $K$-order continuity and a nonexistence of the embedding $E\hookrightarrow{}L^1[0,\infty)$, for a symmetric space $E$ on $[0,\infty)$. In view of this result, we present complete criteria for $K$-order continuity in the Orlicz space $L^\psi$. We also establish a key correspondence between decreasing uniform $K$-monotonicity, $K$-order continuity and upper local uniform $K$-monotonicity in symmetric spaces. 
Section 4 is devoted to (compact) local fully $k$-rotundity properties. Namely, using the local approach we characterize an essential connection between local fully $k$-rotundity on $E^+$ the positive cone of all nonnegative elements of a symmetric space $E$ and $E^d$ the positive cone of all nonnegative and decreasing elements of $E$. Next, we discuss an interesting relationship between local fully $k$-rotundity, local uniform rotundity and order continuity. We also research under which condition compact local fully $k$-rotundity concludes upper local uniform $K$-monotonicity and order continuity. 
In section 5, we answer the crucial question whether compact fully $k$-rotundity on $E^d$ the positive cone of all nonnegative and decreasing elements of a symmetric space $E$ implies order continuity and reflexivity of $E$. It is worth mentioning that the above result improves Proposition 1 in \cite{CHK}, a similar problem that was proved under stronger assumption in a Banach space. We also investigate under which criteria decreasing (resp. increasing) uniform $K$-monotonicity follows directly from compact fully $k$-rotundity on $E^d$. In view of the research published in \cite{CerHudzKamMas}, we focus on the interesting question under which condition compact fully $k$-rotundity might only be considered on the positive cone $E^d$. Finally, in the spirit of the characterization given in section 4, we show that compact fully $k$-rotundity on $E^d$ yields $K$-order continuity in a symmetric space $E$.
Section 6 is dedicated to an application of a geometric structure in symmetric spaces to the approximation theory. First, we establish a complete characterization of approximative compactness in symmetric spaces given in terms of reflexivity, strict $K$-monotonicity, upper local uniform $K$-monotonicity and the Kadec-Klee property for global convergence in measure, respectively. Next, we present an equivalent condition for reflexivity of Lorentz spaces $\Gamma_{p,w}$ expressed in terms of a weight function $w$ for $1<p<\infty$. We also discuss auxiliary examples of Lorentz spaces $\Gamma_{p,w}$ that are reflexive and approximatively compact. It is worth noticing that, the final discussion is devoted to the best dominated approximation problem with respect to the Hardy-Littlewood-P\'olya relation in the Orlicz spaces $L^\psi$ and follows from a general description of $K$-order continuity (see \cite{Cies-JAT}).

\section{Preliminaries}

Let $\mathbb{R}$, $\mathbb{R}^+$ and $\mathbb{N}$ be the sets of reals, nonnegative reals and positive integers, respectively.  A mapping $\phi:\mathbb{R}^+\rightarrow\mathbb{R}^+$ is said to be \textit{quasiconcave} if $\phi(t)$ is increasing and $\phi(t)/t$ is decreasing on $\mathbb{R}^+$ and also $\phi(t)=0\Leftrightarrow{t=0}$. We denote by $\mu$ the Lebesgue measure on $I=[0,\alpha)$, where $\alpha =1$ or $\alpha =\infty$, and by $L^{0}$ the set of all (equivalence classes of) extended real valued Lebesgue measurable functions on $I$.
We denote by $S_X$ (resp. $B_X)$ the unit sphere (resp. the closed unit ball) in a Banach space $(X,\norm{\cdot}{X}{})$. A Banach lattice $(E,\Vert \cdot \Vert _{E})$ is said to be a \textit{Banach function space} (or a \textit{K\"othe space}) if it is a sublattice of $L^{0}
$ and holds the following conditions
\begin{itemize}
\item[(1)] If $x\in L^0$, $y\in E$ and $|x|\leq|y|$ a.e., then $x\in E$ and $%
\|x\|_E\leq\|y\|_E$.
\item[(2)] There exists a strictly positive $x\in E$.
\end{itemize}
For simplicity let us use the short symbol $E^{+}={\{x \in E:x \ge 0\}}$. 
An element $x\in E$ is called a \textit{point of order continuity}, shortly $x\in{E_a}$, if for any
sequence $(x_{n})\subset{}E^+$ such that $x_{n}\leq \left\vert x\right\vert 
$ and $x_{n}\rightarrow 0$ a.e. we have $\left\Vert x_{n}\right\Vert
_{E}\rightarrow 0.$ A Banach function space $E$ is said to be \textit{order continuous}, shortly $E\in \left( OC\right)$, if any element $x\in{}E$ is a point of order continuity. A space $E$ is said to be \textit{reflexive} if $E$ and its associate space $E'$ are order continuous. Given a Banach function space $E$ is said to have the \textit{Fatou property} if for all $\left( x_{n}\right)\subset{}E^+$, $\sup_{n\in \mathbb{N}}\Vert x_{n}\Vert
_{E}<\infty$ and $x_{n}\uparrow x\in L^{0}$, then $x\in E$ and $\Vert x_{n}\Vert _{E}\uparrow\Vert x\Vert
_{E}$. A space $E$ has the \textit{semi-Fatou property} $(E\in(s-FP))$ if for any $(x_n)\subset{E^+}$ such that $x_n\uparrow{x}\in{E^+}$, we have $\norm{x_n}{E}{}\uparrow
\norm{x}{}{}$ (see \cite{LinTza,BS}). In the whole paper, we assume that $E$ has the Fatou property, unless we say otherwise. 

A Banach space $X$ is called \textit{rotund} or \textit{strictly convex} if for any $x,y\in{}S_X$ such that $\norm{x+y}{X}{}=2$, we have $x=y$. 
A point $x\in{X}$ is said to be a \textit{point of local uniform rotundity}, shortly a $LUR$ \textit{point}, if for any sequence $(x_n)\subset{X}$ such that $\norm{x_n+x}{X}{}\rightarrow{2}\norm{x}{X}{}$ and $\norm{x_n}{X}{}\rightarrow\norm{x}{X}{}$ we have $\norm{x_n-x}{X}{}\rightarrow{0}$. A Banach space $X$ is said to be \textit{locally uniformly rotund}, shortly $X$ is $LUR$, if every element $x\in{X}$ is a $LUR$ point in $X$.
Let $k\in\mathbb{N},$ $k\geq{2}$. We say that a Banach space $X$ is \textit{fully $k$-rotund}, shortly $X$ is $FkR$, if each sequence $(x_n)\subset{S_X}$ such that $\norm{\sum_{i=1}^{k}x_{n,{i}}}{X}{}\rightarrow{k}$ for any its $k$-subsequences $(x_{n,{1}}),(x_{n,{2}}),\cdots,(x_{n,{k}})$, is a Cauchy sequence. A Banach space $X$ is said to be \textit{compactly fully $k$-rotund}, shortly $X$ is $CFkR$, if each sequence $(x_n)\subset{S_X}$ such that $\norm{\sum_{i=1}^{k}x_{n,{i}}}{X}{}\rightarrow{k}$ for any its $k$-subsequences $(x_{n,{1}}),(x_{n,{2}}),\cdots,(x_{n,{k}})$, forms a relatively compact set. In case when $k=2$ and a Banach space $X$ is $F2R$ (resp. $CF2R$), then we say that $X$ is \textit{fully rotund} or $FR$ (resp. \textit{compactly fully rotund} or $CFR$). A point $x\in{S_X}$ is said to be a \textit{point of local fully $k$-rotundity}, shortly $x$ is {a point of $LFkR$}, (resp. a \textit{point of compact local fully $k$-rotundity}, shortly $x$ is a point of $CLFkR$) if for each sequence $(x_n)\subset{S_X}$ such that $\norm{x+\sum_{i=1}^{k}x_{n,i}}{X}{}\rightarrow{k+1}$ (see \cite{LinTza,BS}) for any its $k$-subsequences $(x_{n,1})$,$(x_{n,2})$,$\cdots,(x_{n,k})$, we have $x_n$ converges to $x$ in $X$ (resp. $(x_n)$ forms a relatively compact set). A Banach space $X$ is called \textit{locally fully $k$-rotund}, shortly $X$ is $LFkR$, (resp. \textit{compactly locally fully $k$-rotund}, shortly $X$ is $CLFkR$) if every point $x\in{S_X}$ is a point of $LFkR$ (resp. a point of $CLFkR$) (see \cite{FaGl,CHK,HuKoLe}). 

A point $x\in{E}$ is called an $H_g$ point (resp. $H_l$ point) in $E$ if for any $(x_n)\subset{E}$ such that $x_n\rightarrow{x}$ globally (resp. locally) in measure and $\left\Vert x_n\right\Vert _{E}\rightarrow\left\Vert x\right\Vert _{E}$, we have $\left\Vert x_n-x\right\Vert _{E}\rightarrow{0}$. We say that the space $E$ has the Kadec-Klee property globally (resp. locally) in measure if each $x\in{E}$ is an $H_g$point (resp. $H_l$ point) in $E$. A Banach space $E$ has the Kadec-Klee property if for any $(x_n)\subset{E}$ and for any $f$ in the dual space ${E^*}$ of $E$, we have
\begin{equation*}
f(x_n)\rightarrow{f(x)}\quad\textnormal{and}\quad\norm{x_n}{E}{}\rightarrow\norm{x}{E}{}\quad\Rightarrow\quad\norm{x_n-x}{E}{}.
\end{equation*}
The \textit{distribution function} for any function $x\in L^{0}$ is defined by 
\begin{equation*}
d_{x}(\lambda) =\mu\left\{ s\in [ 0,\alpha) :\left\vert x\left(s\right) \right\vert >\lambda \right\},\qquad\lambda \geq 0.
\end{equation*}
For any function $x\in L^{0}$ its \textit{decreasing rearrangement} is given by 
\begin{equation*}
x^{\ast }\left( t\right) =\inf \left\{ \lambda >0:d_{x}\left( \lambda
\right) \leq t\right\}, \text{ \ \ } t\geq 0.
\end{equation*}
In this article we use the notation $x^{*}(\infty)=\lim_{t\rightarrow\infty}x^{*}(t)$ if $\alpha=\infty$ and $x^*(\infty)=0$ if $\alpha=1$. For any function $x\in L^{0}$ we denote the \textit{maximal function} of $x^{\ast }$ by 
\begin{equation*}
x^{\ast \ast }(t)=\frac{1}{t}\int_{0}^{t}x^{\ast }(s)ds.
\end{equation*}
It is well known that for any point $x\in L^{0}$,  $x^{\ast }\leq x^{\ast \ast },$ $x^{\ast \ast }$ is decreasing, continuous and subadditive. For more details of $d_{x}$, $x^{\ast }$ and $x^{\ast \ast }$ see \cite{BS, KPS}. 

We say that two functions $x,y\in{L^0}$ are \textit{equimeasurable}, shortly $x\sim y$, if $d_x=d_y$. A Banach function space $(E,\Vert \cdot \Vert_{E}) $ is called \textit{symmetric} or \textit{rearrangement invariant} (r.i. for short) if whenever $x\in L^{0}$ and $y\in E$ such that $x \sim y,$ then $x\in E$ and $\Vert x\Vert_{E}=\Vert y\Vert _{E}$. The \textit{fundamental function} $\phi_E$ of a symmetric space $E$ we define as follows $\phi_{E}(t)=\Vert\chi_{(0,t)}\Vert_{E}$ for any $t\in [0,\alpha)$ (see \cite{BS}). For any two functions $x,y\in{}L^{1}+L^{\infty }$ the \textit{Hardy-Littlewood-P\'olya relation} $\prec$ is defined by 
\begin{equation*}
x\prec y\quad\Leftrightarrow\quad x^{**}(t)\leq y^{**}(t)\quad\text{ for
all}\quad t>0.\text{ }
\end{equation*}
In this paper we denote the cone of all decreasing rearrangements of elements in symmetric space $E$ by $E^d=\{x^*:x\in{E}\}$.

Now let us introduce shortly $K$-monotonicity properties.
Given a symmetric space $E$ is called $K$-\textit{monotone}, shortly $E\in(KM)$,
if for any $x\in L^{1}+L^{\infty}$ and $y\in E$ with $x\prec y,$ then $x\in E$ and $\Vert x\Vert_{E}\leq \Vert y\Vert _{E}.$ Recall that a symmetric space $E$ is $K$-monotone if and only if $E$ is exact interpolation space between $L^{1}$ and $L^{\infty }.$ Let us also mention the well known fact that a symmetric space $E$ equipped with an order continuous norm or with the Fatou property is $K$-monotone (for more details see \cite{KPS}).
A point $x\in{E}$ is called a \textit{point of upper $K$-monotonicity} (resp. \textit{point of lower $K$-monotonicity}) for short a $UKM$ \textit{point} (resp. an $LKM$ \textit{point}) of $E$ whenever for each $y\in{E}$, $x^*\neq{y^*}$ with $x\prec{y}$ (resp. with $y\prec{x}$), we have $\norm{x}{E}{}<\norm{y}{E}{}$ (resp. $\norm{y}{E}{}<\norm{x}{E}{}$). Let us also remind that a rearrangement invariant space $E$ is said to be \textit{strictly $K$-monotone}, shortly $E\in(SKM)$, if any element of $E$ is a $UKM$ point or equivalently if any element of $E$ is an $LKM$ point.

Given $x\in{E}$ is said to be a \textit{point of $K$-order continuity} of $E$ if for any sequence $(x_n)\subset{E}$ with $x_n\prec{x}$ and $x_n^*\rightarrow{0}$ a.e. we have $\norm{x_n}{E}{}\rightarrow{0}$. In fact, a symmetric space $E$ is called $K$-\textit{order continuous}, shortly $E\in\left(KOC\right)$, if any element $x$ of $E$ is a point of $K$-order continuity.

An element $x\in{E}$ we call a \textit{point of upper local uniform $K$-monotonicity} of $E$, shortly a $ULUKM$ \textit{point}, if for any $(x_n)\subset{E}$ with $x\prec{x_n}$ for any $n\in\mathbb{N}$ and $\norm{x_n}{E}{}\rightarrow\norm{x}{E}{}$, we have $\norm{x^*-x_n^{*}}{E}{}\rightarrow{0}$. Given $x\in{E}$ is said to be a \textit{point of lower local uniform $K$-monotonicity} of $E$, shortly an $LLUKM$ point, if for any $(x_n)\subset{E}$ with $x_n\prec{x}$ for all $n\in\mathbb{N}$ and  $\norm{x_n}{E}{}\rightarrow\norm{x}{E}{}$, then $\norm{x^*-x_n^{*}}{E}{}\rightarrow{0}$. A symmetric space $E$ is called \textit{upper locally uniformly $K$-monotone}, shortly $E\in(ULUKM)$, (resp. \textit{lower locally uniformly $K$-monotone}, shortly ($E\in(LLUKM)$) whenever any element $x\in{}E$ is a $ULUKM$ point (resp. an $LLUKM$ point). We refer the reader for more information to see \cite{ChDSS,Cies-JAT,CieKolPluSKM,CieLewUKM,hkm-geo-prop}.

Now we recall some notions which have been introduced in \cite{CieLewUKM}, and are in some sense a generalization of uniform monotonicity properties in symmetric spaces. It is worth mentioning that the generalization is obtained by replacing a relation $\leq$ by a weaker relation $\prec$, in definition of uniform monotonicity properties. Let us also notice that the generalization characterizes a completely different geometric structure of symmetric spaces than monotonicity properties. A symmetric space $E$ is said to be \textit{uniformly $K$-monotone}, shortly $E\in(UKM)$, if  for any $(x_n),(y_n)\subset{E}$ such that $x_n\prec{y_n}$ for all $n\in\mathbb{N}$ and $\lim_{n\rightarrow\infty}\norm{x_n}{E}{}=\lim_{n\rightarrow\infty}\norm{y_n}{E}{}<\infty$ we have $\norm{x_n^*-y_n^*}{E}{}\rightarrow{0}$.
A symmetric space $E$ is called \textit{decreasing (resp. increasing) uniformly $K$-monotone}, shortly $E\in(DUKM)$ (resp. shortly $E\in(IUKM)$), if  for any $(x_n),(y_n)\subset{E}$ such that $x_{n+1}\prec{}x_n\prec{y_n}$ for all $n\in\mathbb{N}$ and $\lim_{n\rightarrow\infty}\norm{x_n}{E}{}=\lim_{n\rightarrow\infty}\norm{y_n}{E}{}<\infty$ (resp. $x_n\prec{y_n}\prec{y_{n+1}}$ for every $n\in\mathbb{N}$ and $\lim_{n\rightarrow\infty}\norm{x_n}{E}{}=\lim_{n\rightarrow\infty}\norm{y_n}{E}{}<\infty$), we have $\norm{x_n^*-y_n^*}{E}{}\rightarrow{0}$.

Let us assume that $(P)$ is any global property given in a symmetric space $E$. In case when the similar property holds, but considering only nonnegative and decreasing elements of $E$ in the definition, then we say that $E^d$ the positive cone of all nonnegative and decreasing elements of $E$ has a property $(P)$ or equivalently $E^d\in{(P)}$. Analogously, we define a global property that is only satisfied on $E^{+}$ the positive cone of all nonnegative elements of a Banach function space $E$, i.e. we say $E^+$ has $(P)$ or $E^+\in{(P)}$. It is necessary to mention that a symmetric space $E$ is $(LFkR)^*$ (resp. $E$ is $(CLFkR)^*$) if for each $x\in{S_E}$, $x^*$ is a point of local fully $k$-rotundity (resp. a point of compact local fully $k$-rotundity). Similarly, we say that a Banach function space $E$ is $(LFkR)^+$ (resp. $(CLFkR)^+$) if for any $x\in{S_{E}}$, $\abs{x}{}{}$ is a point of local fully $k$-rotundity (resp. a point of compact local fully $k$-rotundity) in $E$.

Recall that the mapping $\psi:\mathbb{R}\rightarrow[0,\infty]$ is said to be an \textit{Orlicz function} if $\psi$ is nonzero function that is even, convex, continuous and vanishes at zero, $\lim_{\abs{t}{}{}\rightarrow\infty}\psi(t)=\infty$. The function $\psi:\mathbb{R}\rightarrow[0,\infty]$ is called an \textit{$N$-function} (resp. \textit{$N$-function at zero}) if $\psi$ is even, convex, continuous (resp. even, convex, continuous) and
\begin{equation*}
\lim_{t\rightarrow 0}\frac{\psi(t)}{t}=0\quad\textnormal{and}\quad\lim_{t\rightarrow \infty}\frac{\psi(t)}{t}=\infty\qquad\left(\textnormal{resp.}\quad\lim_{t\rightarrow 0}\frac{\psi(t)}{t}=0\right).
\end{equation*}  
We use the parameter
\begin{equation*}
a_\psi=\sup\{t>0:\psi(t)=0\}.
\end{equation*}
We say that an Orlicz function $\psi$ satisfies $\Delta_2$ condition for all $u\in\mathbb{R^+}$, shortly $\psi\in\Delta_2$, if there exists $K>0$ such that for all $u\in\mathbb{R}$ we have $\psi(2u)\leq{K}\psi(u)$. Let us notice that if $\psi\in\Delta_2$, then $a_\psi=0$. For any Orlicz function $\psi$ we define its complementary function $\psi_{_Y}$ on $\mathbb{R}$ in the sense of Young and a convex modular $\rho_\psi$ on $L^0$ by
\begin{equation*}
\psi_{_Y}(u)=\sup_{v>0}\{\abs{u}{}{}v-\psi(v)\}\qquad\textnormal{and}\qquad{\rho_\psi}(x)=\int_{I}\psi(x(t))dt
\end{equation*}
for any $u\in\mathbb{R}$ and for any $x\in{L^0}$, respectively. The Orlicz space $L^\psi$ generated by an Orlicz function $\psi$ is given by 
\begin{equation*}
L^\psi=\left\{x\in L^0:\rho_\psi(\lambda x)<\infty,\textnormal{ for some }\lambda>0 \right\}.
\end{equation*}
It is well known that the Orlicz space $L^\psi$ might be considered as a Banach space equipped with the Luxemburg norm 
\begin{equation*}
\norm{x}{\psi}{}=\inf\left\{\lambda>0:\rho_\psi\left(\frac{x}{\lambda}\right)\leq{1}\right\}
\end{equation*}
or with the equivalent Orlicz norm
\begin{equation*}
\norm{x}{\psi}{o}=\sup\left\{\abs{\int_{I}y(t)x(t)dt}{}{}:\rho_{\psi_{_Y}}(y)\leq{1}\right\}.
\end{equation*}
Recall that the Orlicz space $L^\psi$ is order continuous if and only if an Orlicz function $\psi$ satisfies $\Delta_2$ condition. It is worth mentioning that the Orlicz spaces $L^\psi$ are r.i. Banach function spaces under both the Luxemburg and Orlicz norms (for more details the reader is referred to \cite{BS,CHK,KraRut,KPS}).

For given $0<p<\infty$ and a weight function $w\geq{0}$, we define the Lorentz space $\Lambda_{p,w}$, which is a subspace of $L^0$ such that
\begin{equation*}
\left\Vert x\right\Vert _{\Lambda_{p,w}}=\left( \int_{0}^{\alpha}(x^{\ast }(t))^{p}w(t)dt\right) ^{1/p}<\infty,
\end{equation*}
where $W(t)=\int_{0}^{t}w<\infty$ for any $t\in{I}$ and $W(\infty)=\infty$ in the case when $\alpha=\infty$. It is well known that the spaces $\Lambda_{p,w}$ were introduced by Lorentz in \cite{Loren} and the space $\Lambda_{p,w}$ is a norm space (resp. quasi-norm space) if and only if $w$ is decreasing, see \cite{KamMal} (resp. $W$ holds $\Delta_2$ condition, see \cite{Haaker,KamMal}). It is worth reminding that for any $0<p<\infty$ if $W$ fulfills $\Delta_2$ condition and $W(\infty) = \infty$, then the Lorentz space $\Lambda_{p,w}$ is an order continuous r.i. quasi-Banach function space (see \cite{KamMal}). 

For $0<p<\infty $ and $w\in L^{0}$ a nonnegative weight function we consider the
Lorentz space $\Gamma _{p,w}$, that is a subspace of $L^{0}$ such that 
\begin{equation*}
\Vert x\Vert _{\Gamma _{p,w}}=\norm{x^{**}}{\Lambda_{p,w}}{}=\left( \int_{0}^{\alpha }x^{\ast \ast
	p}(t)w(t)dt\right) ^{1/p}<\infty .
\end{equation*}
Unless we say otherwise, we suppose that $w$ belongs to the class $D_{p}$, i.e. 
\begin{equation*}
W(s):=\int_{0}^{s}w(t)dt<\infty \mathnormal{\ \ \ }\text{\textnormal{and}}%
\mathnormal{\ \ }W_{p}(s):=s^{p}\int_{s}^{\alpha }t^{-p}w(t)dt<\infty
\end{equation*}%
for all $0<s\leq 1 $ if $\alpha =1 $ and for all $0<s<\infty$ otherwise. It is easy to observe that if $w\in{}D_p$, then the Lorentz space $\Gamma _{p,w}$ is nontrivial. Moreover, it is easy to see that $\Gamma _{p,w}\subset \Lambda _{p,w}.$ On the other hand, the following inclusion $\Lambda _{p,w}\subset \Gamma _{p,w}$ holds if and only if $w\in B_{p}$ (see \cite{KMGam}). Let us also recall that $\left( \Gamma _{p,w},\Vert \cdot \Vert_{\Gamma _{p,w}}\right)$ is a r.i. quasi-Banach function space with the Fatou property and were introduced by Calder\'{o}n in \cite{Cal}. It is well known that in the case when $\alpha =\infty $ the Lorentz space $\Gamma _{p,w}$ has order continuous norm if and only if $\int_{0}^{\infty }w\left( t\right) dt=\infty$ (see  \cite{KMGam}). It is also commonly known that by the Lions-Peetre $K$-method (see \cite{KPS}), the space $\Gamma_{p,w}$ is an interpolation space between $L^{1}$ and $L^{\infty }$. For more details about the properties of the spaces $\Lambda_{p,w}$ and $\Gamma _{p,w}$ the reader is referred to \cite{CieKolPan,CieKolPlu,KMGam,KamMal}. 

Let $\mathcal{A}$ be a subset of a Banach space $X$ and let $x\in{X}$. We denote
\begin{equation*}
P_\mathcal{A}(x)=\left\{a\in\mathcal{A}:\norm{x-a}{E}{}=\inf_{b\in\mathcal{A}}\norm{x-b}{E}{} \right\}.
\end{equation*}
The best approximation problem is said to be unique (resp. \textit{proximinal} or \textit{solvable}) if $\Card(P_\mathcal{A}(x))\leq{1}$ (resp. if $P_\mathcal{A}(x)\neq\emptyset$).  
We say that the best approximation problem is \textit{uniquely solvable} if $\Card(P_\mathcal{A}(x))=1$. Given sequence $(x_n)\subset{\mathcal{A}-x}$ is said to be a \textit{minimizing sequence} of $\mathcal{A}-x$ if 
\begin{equation*}
\lim_{n\rightarrow\infty}\norm{x_n}{E}{}=\inf_{a\in\mathcal{A}}\norm{x-a}{E}{}=\dist(x,\mathcal{A}).
\end{equation*}
A convex subset $C$ of a Banach space $X$ is called \textit{approximatively compact} if for any $x\in{X}$ and for any sequence $(x_n)\subset{C}$ such that $\norm{x_n-x}{X}{}\rightarrow{d(x,C)}=\inf_{y\in C}\norm{y-x}{X}{}$ we have $(x_n)$ has a Cauchy subsequence. A Banach space $X$ is said to be \textit{approximatively compact} if any closed and convex subset in $X$ is approximatively compact (see \cite{HuKoLe}). 


\section{$K$-order continuity in symmetric spaces}

In this section we show a relation between $K$-order continuity and the embedding of a symmetric space $E$ in $L^1[0,\infty)$. We start our research with the auxiliary lemma.

\begin{lemma}\label{lem:1:KOC}
Let $E$ be a symmetric space on $I=[0,\infty)$. If $x\in{E}$ is a point of order continuity in $E$ and there exists $y\in{E}\setminus{L^1[0,\infty)}$, then $x$ is a point of $K$-order continuity.
\end{lemma}

\begin{proof}
Let us assume that $\phi$ is the fundamental function of $E$. Since $y\in{}E$, by symmetry of $E$ and by Proposition 5.9 \cite{BS} we have for any $t>0$,
\begin{equation*}
\frac{\int_{0}^{t}y^*}{\norm{y}{E}{}}\leq\frac{t}{\phi(t)}.
\end{equation*}
In consequence, since $\int_{0}^{\infty}y^*=\infty$, by Corollary 5.3 \cite{BS} it follows that 
\begin{equation*}
\sup_{t>0}\frac{t}{\phi(t)}=\lim_{t\rightarrow\infty}\frac{t}{\phi(t)}=\infty.
\end{equation*}
Finally, by Proposition 3.1 \cite{CieLewUKM} we conclude $x$ is a point of $K$-order continuity.
\end{proof}

\begin{remark}\label{rem:FP&phi}
	Assuming that $\phi$ is the fundamental function of a symmetric space $E$ with the Fatou property, it is easy to see that $\phi(\infty)=\lim_{t\rightarrow\infty}\phi(t)=\infty$ if and only if $x^*(\infty)=0$ for all $x\in{E}$. 
\end{remark}

\begin{theorem}\label{thm:KOC}
Let $E$ be a symmetric space with the fundamental function $\phi$ on $I=[0,\infty)$. Then the following conditions are equivalent.
\begin{itemize}
	\item[$(i)$] $E$ is order continuous and is not embedded in ${L^1[0,\infty)}$.
	\item[$(ii)$] $E$ is order continuous and $E'\hookrightarrow\{f:f^*(\infty)=0\}$.
	\item[$(iii)$] $E$ is $K$-order continuous and $\phi(\infty)=\infty$.
\end{itemize}
\end{theorem}

\begin{proof}
Assume that $\psi$ is the fundamental function of the associate space $E'$ of a symmetric space $E$. $(i)\Rightarrow(iii)$. Immediately, by Lemma 2.5 in \cite{CieKolPan} and Lemma \ref{lem:1:KOC} and Remark \ref{rem:FP&phi} we get $E$ is $K$-order continuous and $\phi(\infty)=\infty$.\\
$(iii)\Rightarrow(ii)$. By Lemma 3 in \cite{Cies-SKM&KOC} and by Theorem 5.2 in \cite{BS} we have
\begin{equation*}
\lim_{t\rightarrow\infty}\psi(t)=\lim_{t\rightarrow\infty}\frac{t}{\phi(t)}=\infty.
\end{equation*}
Hence, by Corollary 2 in \cite{Cies-SKM&KOC} and by Remark \ref{rem:FP&phi} it follows that $E$ is order continuous and $E'\hookrightarrow\{f:f^*(\infty)=0\}$.\\
$(ii)\Rightarrow(i)$. By Proposition 4.2 in \cite{DSS} under the assumption that $E'\hookrightarrow\{f:f^*(\infty)=0\}$ and by Corollary 1 in \cite{Cies-SKM&KOC} we have $E$ is $K$-order continuous. Next, by Lemma 3 in \cite{Cies-SKM&KOC} we conclude that $E$ is not embedded in ${L^1[0,\infty)}$.
\end{proof}

Now we present the complete characterization of decreasing uniform $K$-monotonicity in a symmetric space $E$, under the assumption that $\phi(\infty)=\infty$, where $\phi$ is the fundamental function of $E$.

\begin{theorem}\label{thm:DUKM<=>KOC&ULUKM}
Let $E$ be a symmetric space with the fundamental function $\phi$ such that $\phi(\infty)=\infty$ in case when $I=[0,\infty)$. The following conditions are equivalent.
\begin{itemize}
	\item[$(i)$] $E$ is decreasing uniformly $K$-monotone.
	\item[$(ii)$] $E$ is $K$-order continuous and upper locally uniformly $K$-monotone.
\end{itemize}
\end{theorem}

\begin{proof}
$(ii)\Rightarrow(i)$. Immediately, in view of assumption that $\phi(\infty)=\infty$ in case when $I=[0,\infty)$, by Theorem 4.13 in \cite{CieLewUKM} and by Remark \ref{rem:FP&phi} we get the wanted implication.\\	
$(i)\Rightarrow(ii)$. First, by Remark 4.1 in \cite{CieLewUKM} it follows that $E$ is upper locally uniformly $K$-monotone. Next, by Corollary 2 in \cite{Cies-SKM&KOC}, in case when $I=[0,1)$, and by Theorem \ref{thm:KOC}, in case when $I=[0,\infty)$, in view of Proposition 4.11 in \cite{CieLewUKM} it is enough to prove in case when $I=[0,\infty)$ that $E$ is not embedded in $L^1[0,\infty)$. Assume for a contrary that $E\overset{C}{\hookrightarrow}{L^1[0,\infty)}$. Then, making analogous observation as in the proof of Lemma 3 in \cite{Cies-SKM&KOC} we have
\begin{equation*}
\lim_{t\rightarrow\infty}\frac{\phi(t)}{t}=d>0.
\end{equation*}
Define for any $n\in\mathbb{N},$
\begin{equation*}
x_n=\frac{1}{2n}\chi_{[0,2n)}\quad\textnormal{and}\quad{y_n}=\frac{1}{n}\chi_{[0,n)}.
\end{equation*}
Clearly, $x_{n+1}\prec{x_n}\prec{}y_n$ for all $n\in\mathbb{N}$ and
\begin{equation*}
\lim_{n\rightarrow\infty}\norm{x_n}{E}{}=\lim_{n\rightarrow\infty}\norm{y_n}{E}{}=d.
\end{equation*} 
Hence, by assumption that $E$ is decreasing uniformly $K$-monotone we obtain 
\begin{equation*}
\norm{y_n^*-x_n^*}{E}{}\rightarrow{0}.
\end{equation*}
On the other hand, by symmetry of $E$ we observe that
\begin{equation*}
\norm{y_n^*-x_n^*}{E}{}=\norm{\frac{1}{2n}\chi_{[0,2n)}}{E}{}=\frac{\phi(2n)}{2n}\rightarrow{d},
\end{equation*}
which implies a contradiction and ends the proof.
\end{proof}

We discuss a complete criteria for $K$-order continuity in the Orlicz space $L^\psi$. 

\begin{theorem}\label{prop:Orlicz:KOC}
Let $\psi$ be an Orlicz function and $\Phi_{L^\psi}$ be the fundamental function of the Orlicz space $L^\psi$. Then, the following assertions are satisfied.
\begin{itemize}
	\item[$(i)$] In case when $I=[0,\infty)$, $a_\psi=0$ if and only if $\Phi_{L^\psi}(\infty)=\infty$.
	\item[$(ii)$] The Orlicz space $L^{\psi}$ is $K$-order continuous and if $I=[0,\infty)$, $\Phi_{L^\psi}(\infty)=\infty$ if and only if $\psi$ satisfies $\Delta_2$ condition and if $I=[0,\infty)$, then $\psi$ is the $N$-function at zero.
\end{itemize}
\end{theorem}

\begin{proof}
$(i)$. Notice that, by Remark \ref{rem:FP&phi}, it is enough to prove that in case when $I=[0,\infty)$, $a_\psi=0$ if and only if $x^*(\infty)=0$ for any $x\in{L^\psi}$. First, assume that there is $x\in{L^\psi}$ with $x^*(\infty)>0$. Then, there exists $\lambda>0$ such that 
\begin{equation*}
\int_{0}^{\infty}\psi(x^*(\infty)\lambda)\leq\int_{0}^{\infty}\psi(x^*\lambda)=I_\psi(x^*\lambda)<\infty.
\end{equation*}
In consequence, since $x^*(\infty)\lambda>0$, we have $\psi(x^*(\infty)\lambda)=0$, which concludes $a_\psi>0$. Conversely, suppose that $a_\psi>0$. Then, letting $a\in(0,a_\psi)$ and $x=a\chi_{(0,\infty)}$ we get $x^*(\infty)>0$ and $x\in{L^\psi}$, which completes the proof.\\
$(ii)$. Immediately, by Corollary 2 in \cite{Cies-SKM&KOC}  and Theorem 10.3 in \cite{KraRut} we may restrict ourselves to the case when $I=[0,\infty)$. First, by Theorem 8.1 in \cite{KraRut} we get the Orlicz space $L^\psi$ is not embedded in $L^1$ if and only if $$\sup_{t>0}\left\{\frac{t}{\psi(t)}\right\}=\infty.$$ Next, by monotonicity of the mapping ${\psi(t)}/{t}$ we obtain the equivalent condition
$$\lim_{t\rightarrow{0^+}}\left\{\frac{\psi(t)}{t}\right\}=\inf_{t>0}\left\{\frac{\psi(t)}{t}\right\}=0,$$
which means that $\psi$ is the $N$-function at zero. Finally, according to Theorem 10.3 in \cite{KraRut} and Theorem \ref{thm:KOC} we complete the proof.
\end{proof}

\section{Local fully $k$-rotundity}

In this section we discuss local fully $k$-rotundity in symmetric spaces. We start our investigation with the well known Theorem 1 in \cite{FaGl}.

\begin{theorem}\label{thm:remark:1}
Let $X$ be a Banach space. The following condition are equivalent.
\begin{itemize}
	\item[$(i)$] $X$ is fully $k$-rotund (resp. compactly fully $k$-rotund).
	\item[$(ii)$] If a sequence $(x_n)\subset{X}$ and $\norm{x_n}{X}{}\rightarrow{1}$,  $\norm{\sum_{i=1}^{k}x_{n,i}}{X}{}\rightarrow{k}$ for any its subsequences $(x_{n,1}),(x_{n,2}),\cdots,(x_{n,k})$, then $(x_n)$ is a Cauchy sequence (resp. $(x_n)$ forms a relatively compact set). 
\end{itemize}
\end{theorem}
Analogously, we may find an equivalent condition for compact local fully $k$-rotundity in a Banach space $X$.
\begin{theorem}\label{thm:remark:2}
	Let $X$ be a Banach space. The following condition are equivalent.
	\begin{itemize}
		\item[$(i)$] $X$ is locally fully $k$-rotund (resp. compactly locally fully $k$-rotund).
		\item[$(ii)$] If for any sequence $(x_n)\subset{X}$ and $x\in{S_X}$, $\norm{x_n}{X}{}\rightarrow{1}$,  $\norm{x+\sum_{i=1}^{k}x_{n,i}}{X}{}\rightarrow{k+1}$ for any its subsequences $(x_{n,1}),(x_{n,2}),\cdots,(x_{n,k})$, then $(x_n)$ is convergent to $x$ (resp. $(x_n)$ forms a relatively compact set). 
	\end{itemize}
\end{theorem}

\begin{proposition}\label{prop:CLFR=>OC}
	Let $E$ be a Banach function space with the semi-Fatou property and let $k\in\mathbb{N}$, $k\geq{2}$. If $x\in{E^+}$ is a point of compact local fully $k$-rotundity, then $x$ is a point of order continuity.
\end{proposition}

\begin{proof}
	For a contrary we may assume that there exists $(x_n)\subset{E^+}$ such that $x\geq{}x_n\geq{x_{n+1}}$ for all $n\in\mathbb{N}$, $x_n\rightarrow{0}$ a.e. and $\norm{x_n}{E}{}\rightarrow{d}>0.$ Without loss of generality we may suppose that $x\geq{2x_n}$ for all $n\in\mathbb{N}$. Define ${y_n=x-x_n}$ for every $n\in\mathbb{N}$. Clearly, $y_n\uparrow{x}$ a.e. Then, by the semi-Fatou property it follows that $\norm{y_n}{E}{}\uparrow\norm{x}{E}{}$. We may assume that $\norm{x}{E}{}=1$, because otherwise we replace $x$ by $x/\norm{x}{E}{}$ and $y_n$ by $y_n/\norm{x}{E}{}$ for all $n\in\mathbb{N}$. Then, since $x\geq{2}x_n$ for any $n\in\mathbb{N}$, by monotonicity of the norm in $E$ for any $k$-subsequences $(y_{n,{1}}),\cdots,(y_{n,{k}})$ of $(y_{n})$ we have
	\begin{align*}
    	k+1&\geq\norm{x+\sum_{i=1}^{k}y_{n,{i}}}{E}{}=\norm{(k+1)x-\sum_{i=1}^{k}x_{n,{i}}}{E}{}\\
	   &\geq\norm{(k+1)x-k\max_{1\leq{i}\leq{k}}\{x_{n,{i}}\}}{E}{}\\
	   &=\norm{(k+2)\left(x-\max_{1\leq{i}\leq{k}}\{x_{n,{i}}\}\right)-x+2\max_{1\leq{i}\leq{k}}\{x_{n,{i}}\}}{E}{}\\
	   &\geq(k+2)\norm{x-\max_{1\leq{i}\leq{k}}\{x_{n,{i}}\}}{E}{}-\norm{x-2\max_{1\leq{i}\leq{k}}\{x_{n,{i}}\}}{E}{}.
	\end{align*}
	Moreover, since $\norm{x}{E}{}=1$, by the semi-Fatou property it is easy to notice that 
	\begin{equation*}
	\norm{x-\max_{1\leq{i}\leq{k}}\{x_{n,{i}}\}}{E}{}\rightarrow{1}\quad\textnormal{and}\quad\norm{x-2\max_{1\leq{i}\leq{k}}\{x_{n,{i}}\}}{E}{}\rightarrow{1},
	\end{equation*}
	whence 
	\begin{equation}\label{equ:CLFkR=>OC}
	\norm{x+\sum_{i=1}^{k}y_{n,{i}}}{E}{}\rightarrow{k+1}.
	\end{equation}
	Therefore, by assumption that $x$ is a point of compact local fully $k$-rotundity there exist a subsequence $(y_{n_k})$ of $(y_n)$ and $z\in{E^+}$ such that $y_{n_k}$ converges to $z$ in norm of $E$. So, passing to subsequence and relabelling we may easily observe that $y_n\rightarrow{z}$ a.e., whence $x=z$ a.e. In consequence, we have
	\begin{equation*}
	\norm{x_n}{E}{}=\norm{x-y_n}{E}{}\rightarrow{0},
	\end{equation*}
	which contradicts with assumption that $\norm{x_n}{E}{}\geq{d}>0$ for any $n\in\mathbb{N}$ and completes the proof.
\end{proof}

Immediately, since the semi-Fatou property follows directly from the Fatou property, using the proof of the previous proposition and by the fact that a point of local fully $k$-rotundity is a point of compact local fully $k$-rotundity in Banach spaces we conclude the following result.

\begin{proposition}\label{prop:LFkR+=>OC}
Let $E$ be a Banach function space. If $x\in{E^+}$ is a point of local fully $k$-rotundity, then $x$ is a point of order continuity. Additionally, if $E^+$ is locally fully $k$-rotund, then $E$ is order continuous.
\end{proposition}

Now, we discuss a relationship between $LFkR$ and $(LFkR)^*$ in a symmetric space $E$. 
First, we investigate an equivalent condition for a point of local fully $k$-rotundity in a symmetric space $E$.
  
\begin{theorem}\label{thm:x&x*:LFkR}
	Let $E$ be a symmetric space and $x\in{S_E}$. The following conditions are equivalent.
	\begin{itemize}
		\item[$(i)$] $x$ is point of local fully $k$-rotundity.
		\item[$(ii)$] $\abs{x}{}{}$ is point of local fully $k$-rotundity.
		\item[$(iii)$] $x^*$ is point of local fully $k$-rotundity.
	\end{itemize} 
\end{theorem}

\begin{proof}
$(i)\Rightarrow(ii)$. Let $(x_n)\subset{S_E}$ be such that for any $k$-subsequences $(x_{n,1})$, $(x_{n,2})$,$\cdots,$$(x_{n,k})$ of $(x_n)$ we have
\begin{equation*}
	\norm{\abs{x}{}{}+\sum_{i=1}^{k}x_{n,i}}{E}{}\rightarrow{k+1}.
\end{equation*}
Then, since $\abs{x}{}{}=x\chi_{\{x\geq{0}\}}-x\chi_{\{x<0\}}$, by symmetry of $E$ we easily observe that $\norm{x_{n}(\chi_{\{x\geq{0}\}}-\chi_{\{x<{0}\}})}{E}{}=\norm{x_n}{E}{}$ for any $n\in\mathbb{N}$ and  
\begin{equation*}
\norm{x+\sum_{i=1}^{k}x_{n,i}(\chi_{\{x\geq{0}\}}-\chi_{\{x<{0}\}})}{E}{}=\norm{\abs{x}{}{}+\sum_{i=1}^{k}x_{n,i}}{E}{}\rightarrow{k+1}.
\end{equation*}
Therefore, by condition $(i)$ it follows that
\begin{equation*}
\norm{\abs{x}{}{}-x_n}{E}{}=\norm{x-x_n(\chi_{\{x\geq{0}\}}-\chi_{\{x<{0}\}})}{E}{}\rightarrow{0}.
\end{equation*}
$(ii)\Rightarrow(iii)$. Assume analogously $(x_n)\subset{S_E}$ and for any $k$-subsequences $(x_{n,1})$, $(x_{n,2})$,$\cdots,$$(x_{n,k})$ of $(x_n)$ we get
\begin{equation*}
\norm{{x^*}+\sum_{i=1}^{k}x_{n,i}}{E}{}\rightarrow{k+1}.
\end{equation*}
Next, by condition $(ii)$ and by Proposition \ref{prop:LFkR+=>OC} we obtain $x$ is a point of order continuity. Thus, by Lemma 2.5 in \cite{CieKolPan} we have $x^*(\infty)=0$. In consequence, by Ryff's theorem in \cite{BS} there exists a measure preserving transformation $\sigma:\supp(x)\rightarrow\supp(x^*)$ such that $x^*\circ\sigma=\abs{x}{}{}$ a.e. In case when $\mu(\supp(x))<\infty$, without loss of generality we may assume that $\sigma:I\rightarrow{I}$ (for more details see \cite{Royd}). Then, by symmetry of $E$ we have $\norm{x_n\circ\sigma}{E}{}=\norm{x_n}{E}{}$ for all $n\in\mathbb{N}$ and
\begin{equation*}
\norm{\abs{x}{}{}+\sum_{i=1}^{k}x_{n,i}\circ\sigma}{E}{}=\norm{{x^*}+\sum_{i=1}^{k}x_{n,i}}{E}{}\rightarrow{k+1}.
\end{equation*}
Hence, by $(ii)$ it follows that
\begin{equation*}
\norm{x^*-x_n}{E}{}=\norm{\abs{x}{}{}-x_n\circ\sigma}{E}{}\rightarrow{0}.
\end{equation*}
$(iii)\Rightarrow(i)$. Since $x^*\in{S_E}$ is a point of local fully $k$-rotundity, by Proposition \ref{prop:LFkR+=>OC} and in view of Lemma 2.6 in \cite{CieKolPan} we conclude that $x$ is a point of order continuity. Let $(x_n)\subset{S_E}$ be such that for any $k$-subsequences $(x_{n,1}),$ $(x_{n,2})$,$\cdots,$$(x_{n,k})$ of $(x_n)$ we have
\begin{equation}\label{equ:1:thm:E*LFkR}
\norm{x+\sum_{i=1}^{k}x_{n,i}}{E}{}\rightarrow{k+1}.
\end{equation}
Then, by subadditivity of the maximal function it is easy to see that
\begin{equation*}
x+\sum_{i=1}^{k}x_{n,i}\prec{x^*+\sum_{i=1}^{k}x_{n,i}^*},
\end{equation*}
whence, by symmetry and by the triangle inequality of the norm in $E$ we obtain
\begin{equation*}
\norm{x^*+\sum_{i=1}^{k}x_{n,i}^*}{E}{}\rightarrow{k+1}.
\end{equation*}
Thus, by assumption $(iii)$ we have
\begin{equation}\label{equ:2:thm:E*LFkR}
\norm{x_n^*-x^*}{E}{}\rightarrow{0}.
\end{equation}
Moreover, by the triangle inequality of the norm in $E$ it is easy to notice that
\begin{equation*}
\frac{1}{2}\norm{x+\sum_{i=1}^{k}x_{n,i}}{E}{}-\frac{1}{2}\sum_{i=1}^{k-1}\norm{x_{n,i}}{E}{}\leq\frac{1}{2}\norm{x+x_{n,k}}{E}{}\leq{1}.
\end{equation*}
Hence, since $(x_n)\subset{S_E}$, replacing a subsequence $(x_{n,k})$ by a sequence $(x_n)$ and denoting $y_n=(x+x_n)/2$ for all $n\in\mathbb{N}$,  by \eqref{equ:1:thm:E*LFkR} we get 
\begin{equation*}
1\geq\norm{y_n}{E}{}=\frac{1}{2}\norm{x+x_{n}}{E}{}\rightarrow{1}.
\end{equation*}
Then, for any $k$-subsequences $(y_{n,1}),$ $(y_{n,2})$,$\cdots,$$(y_{n,k})$ of $(y_n)$ we obtain
\begin{align*}
k+1\geq\norm{x+\sum_{i=1}^{k}y_{n,i}}{E}{}&=\frac{1}{2}\norm{(k+2)x+\sum_{i=1}^{k}x_{n,i}}{E}{}\\
&\geq\frac{1}{2}\norm{(k+2)\left(x+\sum_{i=1}^{k}x_{n,i}\right)-(k+1)\sum_{i=1}^{k}x_{n,i}}{E}{}\\
&\geq\frac{k+2}{2}\norm{x+\sum_{i=1}^{k}x_{n,i}}{E}{}-\frac{k+1}{2}\sum_{i=1}^{k}\norm{x_{n,i}}{E}{}.
\end{align*} 
Therefore, by \eqref{equ:1:thm:E*LFkR} this yields that
\begin{equation*}
\norm{x+\sum_{i=1}^{k}y_{n,i}}{E}{}\rightarrow{k+1}.
\end{equation*}
Consequently, by the inequality 
\begin{equation*}
x+\sum_{i=1}^{k}y_{n,i}\prec{x^*+\sum_{i=1}^{k}y_{n,i}^*},
\end{equation*}
and by symmetry of $E$ and in view of assumption $(iii)$ it follows that
\begin{equation*}
\norm{\left(\frac{x+x_n}{2}\right)^*-x^*}{E}{}=\norm{y_n^*-x^*}{E}{}\rightarrow{0}.
\end{equation*}
Hence, by \eqref{equ:2:thm:E*LFkR} and by Lemma 2.2 in \cite{CzeKam} it follows that $x_n\rightarrow{x}$ globally in measure. Finally, since $x$ is a point of order continuity, by \eqref{equ:2:thm:E*LFkR} and by Proposition 2.4 in \cite{CzeKam} we finish the proof.
\end{proof}

The immediate consequence of Theorem \ref{thm:x&x*:LFkR} are the following results. 

\begin{theorem}
	Let $E$ be a symmetric space. The following are equivalent.
	\begin{itemize}
		\item[$(i)$] $E$ is locally fully $k$-rotund.
		\item[$(ii)$] $E$ is $(LFkR)^+$.
		\item[$(iii)$] $E$ is $(LFkR)^*$.
	\end{itemize}
\end{theorem}

\begin{theorem}
	Let $E$ be a symmetric space and let;
	\begin{itemize}
		\item[$(i)$] $E$ is locally fully $k$-rotund.
		\item[$(ii)$] $E^+$ is locally fully $k$-rotund.
		\item[$(iii)$] $E^d$ is locally fully $k$-rotund.
	\end{itemize} 
	Then, $(i)\Leftrightarrow(ii)\Rightarrow(iii)$. If $E$ is order continuous, then $(iii)\Rightarrow(i)$.
\end{theorem}

\begin{proof}
	$(i)\Rightarrow(ii)\Rightarrow(iii)$. It is obvious. \\
	$(iii)\Rightarrow(i)$. We proceed analogously as in the proof of Theorem \ref{thm:x&x*:LFkR} under the assumption that $E$ is order continuous.\\
	$(ii)\Rightarrow(i)$. Immediately, by Proposition \ref{prop:LFkR+=>OC} we get $E$ is order continuous. Finally, since $(ii)\Rightarrow(iii)$ and $(iii)\Rightarrow(i)$ we complete the proof. 
\end{proof}

Now, we investigate a correspondence between local uniform rotundity and local fully $k$-rotundity in Banach spaces.  

\begin{theorem}
	Let $X$ be a Banach space. If $X$ is locally uniformly rotund, then $X$ is locally fully $k$-rotund.
\end{theorem}

\begin{proof}
	Let $k\in\mathbb{N}$, $k\geq{2}$, and let $(x_n)\subset{X}$ and $x\in{S_X}$ be such that for any $k$-subsequences $(x_{n,{1}}),(x_{n,{2}}),\cdots,(x_{n,{k}})$ of a sequence $(x_n)$,
	\begin{equation}\label{equ:1:thm:LUR}
	\norm{\sum_{i=1}^{k}x_{n,{i}}+x}{X}{}\rightarrow{k+1}\quad\textnormal{and}\quad\norm{x_n}{X}{}\rightarrow{1}.
	\end{equation}
	By the triangle inequality of the norm in $X$ we notice that
	\begin{align*}
	\frac{1}{2}\norm{\sum_{i=1}^{k}x_{n,{i}}+x}{X}{}-\frac{1}{2}\sum_{i=2}^{k}\norm{x_{n,{i}}}{X}{}
	\leq&\frac{1}{2}\norm{\sum_{i=1}^{k}x_{n,{i}}+x}{X}{}-\frac{1}{2}\norm{\sum_{i=2}^{k}x_{n,{i}}}{X}{}\\
	\leq&\frac{1}{2}\norm{x_{n,{1}}+x}{X}{}\leq{1}
	\end{align*}
	for any $k\in\mathbb{N}$. Then, for any subsequence $(x_{n,{1}})$ of $(x_n)$, by \eqref{equ:1:thm:LUR} we have 
	\begin{equation}\label{equ:2:thm:LUR}
	\norm{x_{n,{1}}+x}{X}{}\rightarrow{2}\quad\textnormal{and}\quad\norm{x_{n,{1}}}{X}{}\rightarrow{1}.
	\end{equation}
	Define for all $n\in\mathbb{N},$
	\begin{equation*}
	u_{n}=\frac{x_{n}}{\norm{x_{n}}{X}{}}.
	\end{equation*}
	Then, $(u_{n})\subset{S_X}$ and by \eqref{equ:1:thm:LUR} we get
	\begin{equation}\label{equ:3:thm:LUR}
	\norm{u_{n}-x_{n}}{X}{}=\norm{x_{n}}{X}{}\abs{1-\frac{1}{\norm{x_{n}}{X}{}}}{}{}\rightarrow{0}.
	\end{equation}
	Next, passing to subsequence $(u_{n,{1}})$, by the triangle inequality of the norm in $X$ and by \eqref{equ:2:thm:LUR} and \eqref{equ:3:thm:LUR} it follows that
	\begin{equation*}
	\norm{u_{n,{1}}+x}{X}{}\rightarrow{2}.
	\end{equation*}
	Thus, by assumption that $X$ is locally uniformly rotund and by \eqref{equ:3:thm:LUR} we obtain
	\begin{equation*}
	\norm{x_{n,{1}}-x}{X}{}\rightarrow{0}.
	\end{equation*}
	Finally, since $(x_{n,{1}})$ is arbitrary chosen subsequence of $(x_n)$ we get the end of the proof.
\end{proof}

\begin{theorem}\label{thm:CLFkR*=>ULUKM}
	Let $E$ be a symmetric space. If $E^d$ is compactly locally fully $k$-rotund and strictly $K$-monotone, then $E$ is upper locally uniformly $K$-monotone.
\end{theorem}

\begin{proof}
	Let $(x_n)\subset{E}$ and $x\in{E}$ be such that $x\prec{x_n}$ for all $n\in\mathbb{N}$, $\norm{x_n}{E}{}\rightarrow\norm{x}{E}{}$. Without loss of generality we may assume that $\norm{x}{E}{}=1$. Then, we have
	\begin{equation}\label{equ:1:thm:ULUKM}
	\norm{x_n}{E}{}\rightarrow{1}.
	\end{equation}
	Moreover, since $x\prec{x_n}$ for all $n\in\mathbb{N}$, for any $k$-subsequences $(x_{n,{1}}),(x_{n,{2}}),\cdots,(x_{n,{k}})$ of $(x_n)$ we get
	\begin{equation*}
	(k+1)x\prec{x^*+\sum_{i=1}^{k}x_{n,{i}}^*}.
	\end{equation*}
	Hence, by symmetry and by the triangle inequality of the norm of $E$ and by \eqref{equ:1:thm:ULUKM} it follows that
	\begin{equation*}
	\norm{x^*+\sum_{i=1}^{k}x_{n,{i}}^*}{E}{}\rightarrow{k+1}.
	\end{equation*}
	In consequence, by assumption that $E^d$ is compactly locally fully $k$-rotund, there exists a subsequence $(x_{n_k})$ of $(x_n)$ such that $x_{n_k}^*$ converges to $y\in{E}$ in norm of $E$. Hence, by Lemma 3.2 in \cite{KPS} we get $y=y^*$ a.e. and
	$$\norm{x_{n_k}^*-y^*}{E}{}\rightarrow{0}.$$ 
	Therefore, by Proposition 5.9 in \cite{BS} we have for all $t>0,$
	\begin{equation*}
	x_{n_k}^{**}(t)\rightarrow{y^{**}(t)}.
	\end{equation*}
	Thus, since $x\prec x_n$ for any $n\in\mathbb{N}$ we obtain $x\prec{y}$. Finally, since $\norm{x}{E}{}=\norm{y}{E}{}=1$, by assumption that $E$ is strictly $K$-monotone we get $x^*=y^*$ a.e. and so by the double extract sequence theorem we conclude $$\norm{x_n^*-x^*}{E}{}\rightarrow{0}.$$
\end{proof}

\section{Fully $k$-rotundity}

In this section we investigate a relationship between (compact) fully $k$-rotundity, decreasing (increasing) uniform $K$-monotonicity and reflexivity in a symmetric space $E$. First, we show a connection between compact fully $k$-rotundity and order continuity in $E$.  

\begin{proposition}\label{prop:1:CFR*=>OC}
	Let $E$ be a symmetric space. If $E^d$ is compactly fully $k$-rotund, then $E$ is order continuous.
\end{proposition}

\begin{proof}
	First we prove that $\phi(\infty)=\infty$. In view of Remark \ref{rem:FP&phi}, we may assume for a contrary that there exists $x\in{E}$ such that $x^*(\infty)>0$. Then, we easily observe that $L^\infty\hookrightarrow{E}$ with some constant $C>0$. Define $x=\chi_{[0,\infty)}$ and $x_n=\chi_{[0,n)}$ for every $n\in\mathbb{N}$. Clearly, we have $x,x_n\in{E}$ for all $n\in\mathbb{N}$, $x_n\uparrow{x}$ a.e. and $\sup_{n\in\mathbb{N}}\norm{x_n}{E}{}\leq\norm{x}{E}{}<\infty$. Hence, by the Fatou property we conclude $\norm{x_n}{E}{}\uparrow\norm{x}{E}{}$. Denote $v_n=x_n/\norm{x_n}{E}{}$ for all $n\in\mathbb{N}$. Let $(v_{n,{1}}),\cdots,(v_{n,{k}})$ be any $k$-subsequences of $(v_n)$. Then, by symmetry of $E$ we get
	\begin{align*}
	k\geq\norm{\sum_{i=1}^{k}v_{n,{i}}}{E}{}&\geq\min_{1\leq{i}\leq{k}}\{\norm{x_{n,{i}}}{E}{}{}\}\abs{\sum_{i=1}^{k}\frac{1}{\norm{x_{n,{i}}}{E}{}}}{}{}\\
	&\geq\min_{1\leq{i}\leq{k}}\{\norm{x_{n,{i}}}{E}{}{}\}\frac{k}{\norm{x}{E}{}}
	\end{align*}
	for all $n\in\mathbb{N}$, whence 
	\begin{equation}\label{equ:CFR*=>OC}
	\norm{\sum_{i=1}^{k}v_{n,{i}}}{E}{}\rightarrow{}k.
	\end{equation}
	Consequently, since $v_n=v_n^*$ for any $n\in\mathbb{N}$ and by assumption that $E^d$ is compactly fully $k$-rotund, it follows that $(v_n)$ forms a relatively compact set. Therefore, passing to subsequence and relabelling if necessary we may suppose that $v_n$ converges to $v\in{S_E}$ in norm of $E$ as well as a.e. on $I$. Thus, since $v_n\rightarrow{x}/{\norm{x}{E}{}}$ a.e. we conclude $v=x/\norm{x}{E}{}$ a.e. Moreover, by the triangle inequality of the norm in $E$ we obtain
	\begin{equation*}
	\frac{\norm{x_n-x}{E}{}}{\norm{x_n}{E}{}}\leq\norm{v_n-v}{E}{}+\norm{x}{E}{}\abs{\frac{1}{\norm{x_n}{E}{}}-\frac{1}{\norm{x}{E}{}}}{}{}
	\end{equation*}
	for every $n\in\mathbb{N}$. In consequence, we have $$\norm{x_n-x}{E}{}\rightarrow{0}.$$ On the other hand, by symmetry of $E$ it is easy to see that $\norm{x_n-x}{E}{}=\norm{x}{E}{}$ for all $n\in\mathbb{N}$, which gives us a contradiction and proves that $\phi(\infty)=\infty$. Now, we show that $E$ is order continuous. Let us assume for a contrary that there exists a sequence $(x_n)\subset{E^+}$ such that $x_n\downarrow{0}$ a.e. and $d=\inf_{n\in\mathbb{N}}\norm{x_n}{E}{}>0.$ Clearly, $x_{n+1}\leq{x_n}$ for any $n\in\mathbb{N}$, whence we have $\norm{x_n}{E}{}\downarrow{d}.$ Define for any $n\in\mathbb{N }$,
	\begin{equation*}
	u_n=\frac{x_n}{\norm{x_n}{E}{}}.
	\end{equation*}
	It is obvious that $(u_n)\subset{S_{E^+}}$ and $u_n\rightarrow{0}$ a.e. Moreover, by symmetry of $E$, for any $k$-subsequences $(u_{n,{1}}),\cdots,(u_{n,{k}})$ of $(u_n)$ we have 
	\begin{align*}
	k\geq\norm{\sum_{i=1}^{k}u_{n,{i}}}{E}{}&\geq\min_{1\leq{i}\leq{k}}\{\norm{x_{n,{i}}}{E}{}\}\abs{\sum_{i=1}^{k}\frac{1}{\norm{x_{n,{i}}}{E}{}}}{}{}\\
	&\geq{d}\abs{\sum_{i=1}^{k}\frac{1}{\norm{x_{n,{i}}}{E}{}}}{}{}\geq\frac{kd}{\max_{1\leq{i}\leq{k}}\{\norm{x_{n,{i}}}{E}{}\}}
	\end{align*}
	for every $n\in\mathbb{N}$. Hence, since for all $n\in\mathbb{N}$,
	$$\sum_{i=1}^{k}u_{n,{i}}\prec\sum_{i=1}^{k}u_{n,{i}}^*$$
	and by assumption that $\norm{x_n}{E}{}\downarrow{d}$ and by symmetry of $E$ it follows that
	\begin{equation*}
	k\geq\norm{\sum_{i=1}^{k}u_{n,{i}}^*}{E}{}\geq\norm{\sum_{i=1}^{k}u_{n,{i}}}{E}{}\rightarrow{k}.
	\end{equation*}
	Thus, since $(u_n^*)\subset{S_{E}}$, by assumption that $E^d$ is compactly fully $k$-rotund we obtain $(u_n^*)$ forms a relatively compact set. Therefore, passing to subsequence and relabelling if necessary we may assume that there exists $u\in{S_E}$ such that
	\begin{equation}\label{equ:1:thm:FR*}
	\norm{u_n^*-u}{E}{}\rightarrow{0}\qquad\textnormal{and}\qquad{}u_n^*\rightarrow{u}\quad\textnormal{a.e.}
	\end{equation}
	On the other hand, since $\phi(\infty)=\infty$, by Remark \ref{rem:FP&phi} we get $x_1^*(\infty)=0$. So, by assumption that $x_n\downarrow{0}$ a.e., in view of Property 2.12 in \cite{KPS} this concludes that $x_n^*\rightarrow{0}$ a.e. Hence, by definition of $u_n$ for any $n\in\mathbb{N}$ we observe $u_n^*\rightarrow{0}$ a.e. Therefore, by \eqref{equ:1:thm:FR*} this provides $u=0$ a.e. Consequently, in view of the fact that $(u_n^*)\subset{S_E}$ and by \eqref{equ:1:thm:FR*} we get a contradiction which completes the proof.
\end{proof}

The next proposition follows directly from the well known result in \cite{CHK}, where there has been shown a complete correspondence between fully $k$-rotundity and compact fully $k$-rotundity as well as rotundity in a Banach lattice. Applying the same technique as in paper \cite{CHK} we may easily show that the below relationships are satisfied on the positive cone $E^d$ of a symmetric space $E$.

\begin{proposition}\label{prop:relationFR&CFR}
Let $E$ be a symmetric space. Then, $E^d$ is fully $k$-rotund if and only if $E^d$ is compactly fully $k$-rotund and rotund.
\end{proposition}

\begin{proposition}\label{prop:relationLFR&CLFR}
	Let $E$ be a symmetric space. Then, $E^d$ is locally fully $k$-rotund if and only if $E^d$ is compactly locally fully $k$-rotund and rotund.
\end{proposition}

Recently, in paper \cite{CHK} authors have proved that if a Banach space $X$ is compactly fully $k$-rotund then $X$ is reflexive. We show that in a symmetric space $E$ it is enough to assume weaker condition to get reflexivity of $E$. For the sake of completeness and reader's convenience we present all details of the proof of the following theorem. In some parts of the proof we use similar technique to the proof of Proposition 1 in \cite{CHK}.

\begin{theorem}\label{thm:CFR=>reflexive}
Let $E$ be a symmetric space. If $E^d$ is compactly fully $k$-rotund, then $E$ is reflexive.
\end{theorem}

\begin{proof}
Let $\phi$ be the fundamental function of a symmetric space $E$. First, we claim 
\begin{equation}\label{equ:claim}
\lim_{t\rightarrow\infty}\frac{t}{\phi(t)}=\infty.
\end{equation}	
Let us suppose that it is not true. Then, by monotonicity of the map $\phi(t)/t$ we get
\begin{equation*}
\lim_{t\rightarrow\infty}\frac{\phi(t)}{t}=d>0.
\end{equation*}
Define for any $n\in\mathbb{N}$ and $t>0$,
\begin{equation*}
x_n(t)=\frac{1}{dn}\chi_{[0,n)}(t).
\end{equation*}
Clearly, $x_{n+1}\prec{x_n}=x_n^*$ for any $n\in\mathbb{N}$ and $\norm{x_n}{E}{}\downarrow{1}$. Next,  for any $k$-subsequences $(x_{n,1}),(x_{n,2}),\cdots,(x_{n,k})$ of $(x_n)$, by symmetry of $E$ we have
\begin{equation*}
k\min_{1\leq{i}\leq{k}}\{\norm{x_{n,i}}{E}{}\}\leq\norm{\sum_{i=1}^{k}x_{n,i}}{E}{}\leq{k}\max_{1\leq{i}\leq{k}}\{\norm{x_{n,i}}{E}{}\},
\end{equation*}
for all $n\in\mathbb{N}$, whence
\begin{equation*}
\norm{\sum_{i=1}^{k}x_{n,i}}{E}{}\rightarrow{k}.
\end{equation*}
Therefore, by Theorem \ref{thm:remark:1} and by assumption that $E^d$ is compactly fully $k$-rotund, passing to subsequence and relabeling if necessary we may assume that there exists $x\in{S_E}$ such that 
\begin{equation}\label{equ:norm:conv}
\norm{x_n-x}{E}{}\rightarrow{0}\quad\textnormal{and}\quad{x_n\rightarrow}x\quad\textnormal{a.e.}
\end{equation} 
On the other hand, by construction of $x_n$ for any $n\in\mathbb{N}$ it is easy to observe that $x_n\rightarrow{0}$ a.e. Hence, $x=0$ a.e. and consequently in view of the fact \eqref{equ:norm:conv} it follows that 
\begin{equation*}
\frac{\phi(n)}{dn}=\norm{x_n}{E}{}\rightarrow{0}.
\end{equation*} 
Therefore, since $\norm{x_n}{E}{}\downarrow{1}$ we get a contradiction which proves our claim \eqref{equ:claim}.\\	
Let $f\in{S_{E^*}}$ be an element of dual space of $E$. Then, there exists a sequence $(x_n)\subset{S_E}$ such that $f(x_n)\rightarrow\norm{f}{E^*}{}=1$. By assumption that $E^d$ is compactly fully $k$-rotund and by Proposition \ref{prop:1:CFR*=>OC} it follows that $E$ is order continuous and also $\phi(\infty)=\infty$. Consequently, by Theorem 4.1 in \cite{BS} this yields that the dual space $E^*$ and the associate space $E'$ of a symmetric space $E$ coincide. Hence, by the Corollary 4.4 in \cite{BS} we have
\begin{equation}\label{equ:1:thm:reflex}
\norm{f}{E^*}{}=\norm{f}{E'}{}=\sup\left\{\int_{0}^{\infty}f^*(t)x^*(t)dt:\hspace{0.05in}\norm{x}{E}{}\leq{1}\right\}.
\end{equation}
Let $\psi$ be the fundamental function of $E'$. By \eqref{equ:claim} and by Theorem 5.2 in \cite{BS} we get $\psi(\infty)=\infty$, whence by Remark \ref{rem:FP&phi} we conclude that $f^*(\infty)=0$. In consequence, by Ryff's theorem in \cite{BS} there exists a measure preserving transformation $\sigma:\supp(f)\rightarrow\supp(f^*)$ such that $f^*\circ\sigma=|f|$ a.e. on $\supp(f)$. First, if $\mu(\supp(f))<\infty$ we may consider $\sigma:I\rightarrow{I}$ (see \cite{Royd}). Moreover, without loss of generality we may assume that  $\supp(x_n)\subset\supp(f)$ for any $n\in\mathbb{N}$. Next, by \eqref{equ:1:thm:reflex} and by assumption that $f(x_n)\rightarrow{1}$, in view of the Hardy-Littlewood inequality in \cite{BS} we obtain
\begin{equation*}
1=\lim_{n\rightarrow\infty}f(x_n)=\lim_{n\rightarrow\infty}\int_{0}^{\infty}f(t)x_n(t)dt=\lim_{n\rightarrow\infty}\int_{0}^{\infty}f^*(t)x_n^*(t)dt.
\end{equation*}
Define for any $n\in\mathbb{N}$,
\begin{equation*}
y_n=\sg(f)x_n^*\circ\sigma\chi_{\sigma^{-1}[I]}.
\end{equation*} 
Then, we have 
\begin{align}\label{equ:conver:to:1}
1=\lim_{n\rightarrow\infty}\int_{0}^{\infty}f^*(t)x_n^*(t)dt&=\lim_{n\rightarrow\infty}\int_{\sigma^{-1}[I]}f^*\circ\sigma(t)x_n^*\circ\sigma(t)dt\\
&=\lim_{n\rightarrow\infty}\int_{0}^{\infty}f(t)\sg(f(t))x_n^*\circ\sigma(t)\chi_{\sigma^{-1}[I]}(t)dt\nonumber\\
&=\lim_{n\rightarrow\infty}f(y_n).\nonumber
\end{align}
Hence, taking any $k$-subsequences $(y_{n,{1}}),\cdots,(y_{n,{k}})$ of $(y_n)$ we conclude
\begin{equation*}
f\left(\sum_{i=1}^{k}y_{n,{i}}\right)\rightarrow{k}.
\end{equation*}
Next, by subadditivity of the maximal function we observe  for any $n\in\mathbb{N}$,
\begin{equation*}
\sum_{i=1}^{k}y_{n,{i}}\prec\sum_{i=1}^{k}y_{n,{i}}^*.
\end{equation*}
In consequence, since $y_n^*=x_n^*$ a.e. for all $n\in\mathbb{N}$ and $f\in{S_{E^*}}$, by symmetry and by the triangle inequality of the norm in $E$ we get
\begin{equation*}
\norm{\sum_{i=1}^{k}x_{n,{i}}^*}{E}{}\rightarrow{k}.
\end{equation*}
Therefore, by assumption that $E^d$ is compactly fully $k$-rotund, passing to subsequence and relabelling if necessary we may suppose that $x_n^*$ converges to $x\in{E}$ in norm of $E$ and also a.e. So, $\norm{x}{E}{}=1$ and by Lemma 3.2 in \cite{KPS} we get $x^*=x$ a.e. Hence, since $\supp(x_n^*)\subset\supp(f^*)$ for any $n\in\mathbb{N}$,  without loss of generality we may assume that $\supp(x^*)\subset\supp(f^*)$. Define $y=\sg(f)x^*\circ\sigma\chi_{\sigma^{-1}[I]}$. Then,  it is easy to see that $y^*=x^*$ a.e.  and $y\in{S_E}$. Moreover, by \eqref{equ:conver:to:1} and by continuity of $f$ we have
\begin{align*}
f(y)=\int_{0}^{\infty}f(t)y(t)dt&=\int_{0}^{\infty}f(t)\sg(f(t))x^*\circ\sigma(t)\chi_{\sigma^{-1}[I]}(t)dt\\
&=\lim_{n\rightarrow\infty}\int_{0}^{\infty}f(t)\sg(f(t))x_n^*\circ\sigma(t)\chi_{\sigma^{-1}[I]}(t)dt\\
&=\lim_{n\rightarrow\infty}f(y_n)=1,
\end{align*}
 which finishes the proof.
\end{proof}

In paper \cite{CerHudzKamMas}, authors have showed among others a relationship between the facts $E^d\in(FR)$ and $E\in(FR)$ in symmetric spaces under the additional assumption that there exists an equivalent symmetric uniformly rotund norm. In the spirit of the previous result we investigate a correlation between $E^d\in(CFkR)$ and $E\in(CFkR)$ in symmetric spaces. For the sake of completeness and reader's convenience we present the proof of the following theorem even though it is similar in some parts to the proof of Theorem 2 in \cite{CerHudzKamMas}.

\begin{theorem}
Let $(E,\norm{\cdot}{E}{})$ be a symmetric space. If $E^d$ is compactly fully $k$-rotund and locally uniformly rotund and also $E$ has an equivalent symmetric norm $\norm{\cdot}{o}{}$ that is compactly fully $k$-rotund, then $E$ is compactly fully $k$-rotund.
\end{theorem}

\begin{proof}
Let $(x_n)\subset{S_E}$ be such that for any $k$-subsequences $(x_{n,1}),\cdots,(x_{n,k})$ of $(x_n)$ we have $\norm{\sum_{i=1}^{k}x_{n,i}}{E}{}\rightarrow{k}.$
Then, since 
\begin{equation}\label{equ:0:CHKM}
\sum_{i=1}^{k}x_{n,i}\prec\frac{\left(\sum_{i=1}^{k}x_{n,i}\right)^*}{2}+\frac{\sum_{i=1}^{k}x_{n,i}^*}{2}\prec\sum_{i=1}^{k}x_{n,i}^*
\end{equation}
for any $k\in\mathbb{N}$, by symmetry and by the triangle inequality of the norm in $E$ we obtain $$\norm{\sum_{i=1}^{k}x_{n,i}^*}{E}{}\rightarrow{k}.$$ Hence, by assumption that $E^d$ is compactly fully $k$-rotund, passing to subsequence and relabelling if necessary we may assume that there exists $x\in{S_E}$ such that
\begin{equation}\label{equ:1:CHKM}
\norm{x_{n}^*-x}{E}{}\rightarrow{0}.
\end{equation}
Thus, by Lemma 3.2 in \cite{KPS} we obtain $x=x^*$ a.e. Next, by \eqref{equ:0:CHKM} and \eqref{equ:1:CHKM} we get 
\begin{equation*}
\norm{\left(\sum_{i=1}^{k}\frac{x_{n,i}}{k}\right)^*+x^*}{E}{}\rightarrow{2}.
\end{equation*}
In consequence, by assumption that $E^d$ is $LUR$ and $E$ has the equivalent symmetric norm $\norm{\cdot}{o}{}$ we conclude 
\begin{equation*}
\norm{\sum_{i=1}^{k}\frac{x_{n,i}}{\norm{x}{o}{}}}{o}{}\rightarrow{k}.
\end{equation*}
Finally, by assumption that $\norm{\cdot}{o}{}$ is compactly fully $k$-rotund, passing to subsequence and relabelling if necessary we may assume that $(x_n)$ is a Cauchy sequence in $E$. 
\end{proof}

\begin{remark}
Let us notice that local uniform rotundity does not imply compact fully $k$-rotundity in symmetric spaces in general. Consider a sequence symmetric space $E=l_{1}$ with an equivalent norm $\norm{\cdot}{E}{}$ given by
\begin{equation*}
\norm{x}{E}{}=\left(\norm{x}{1}{2}+\norm{x}{2}{2}\right)^{1/2}
\end{equation*}
for any $x\in{E}$. By Example 5.3.6 in \cite{Meggin}, it is well known that $E$ is locally uniformly rotund and also $E$ is not reflexive. Next, since the proof of Theorem \ref{thm:CFR=>reflexive} for the sequence case is analogous, it is easy to see that $E^d$ is not compactly fully $k$-rotund.
\end{remark}

\begin{theorem}\label{thm:CFkR*=>DUKM}
Let $E$ be a symmetric space. If $E^d$ is compactly fully $k$-rotund and strictly $K$-monotone, then $E$ is decreasing uniformly $K$-monotone.
\end{theorem}

\begin{proof}
First, by Proposition 4.3 in \cite{CieLewUKM}, we may assume that $(x_n),(y_n)\subset{E}$, $x_{n+1}\prec{}x_n\prec{y_n}$ for every $n\in\mathbb{N}$ and 
\begin{equation}\label{equ:1:thm:FR*=>DIUKM}
\norm{x_n}{E}{}\rightarrow{1}\quad\textnormal{and}\quad\norm{y_n}{E}{}\rightarrow{1}.
\end{equation}
Moreover, for any $k$-subsequences $({x_{n,{1}}}),({x_{n,{2}}}),\cdots,({x_{n,{k}}})$ of $(x_n)$ and for corresponding $k$-subsequences $({y_{n,{1}}}),\cdots,({y_{n,{k}}})$ of $(y_n)$ we have
\begin{equation*}
\sum_{i=1}^{k}{x_{n,{i}}^*}\prec{}\sum_{i=1}^{k}{y_{n,{i}}^*}
\end{equation*}
for any $n\in\mathbb{N}$. Therefore, since $x_{n+1}\prec{x_n}$ for all $n\in\mathbb{N}$, by symmetry of $E$ and by the triangle inequality of the norm in $E$ we get
\begin{align*}
k\min_{1\leq{}i\leq{k}}\left\{\norm{x_{n,{i}}}{E}{}\right\}\leq\norm{\sum_{i=1}^{k}{x_{n,{i}}^*}}{E}{}\leq\norm{\sum_{i=1}^{k}{y_{n,{i}}^*}}{E}{}\leq{k}\max_{1\leq{}i\leq{k}}\left\{\norm{y_{n,{i}}}{E}{}\right\}
\end{align*}
for any $n\in\mathbb{N}$. Thus, by \eqref{equ:1:thm:FR*=>DIUKM} we obtain
\begin{equation*}
\norm{\sum_{i=1}^{k}{x_{n,{i}}^*}}{E}{}\rightarrow{k}\quad\textnormal{and}\quad\norm{\sum_{i=1}^{k}{y_{n,{i}}^*}}{E}{}\rightarrow{k}.
\end{equation*}
In consequence, by assumption that $E^d$ is compactly fully $k$-rotund, there exist some subsequences $(x_{n_j})$ of $(x_n)$ and $(y_{n_j})$ of $(y_n)$ as well as $x,y\in{E}$ such that $x_{n_j}\prec{y_{n_j}}$ for all $j\in\mathbb{N}$ and
\begin{equation}\label{equ:2:thmFR*=>DIUKM}
\norm{x_{n_j}^*-x}{E}{}\rightarrow{0}\quad\textnormal{and}\quad\norm{y_{n_j}^*-y}{E}{}\rightarrow{0}.
\end{equation}
Hence, by \eqref{equ:1:thm:FR*=>DIUKM} we conclude $x,y\in{S_E}$ and also by Lemma 3.2 in \cite{KPS} it follows that $x=x^*$ and $y=y^*$ a.e. Therefore, in view of Proposition 5.9 in \cite{BS} we have
\begin{equation}\label{equ:4:thm:FR*=>DIUKM}
x_{n_j}^{**}(t)\rightarrow{x}^{**}(t)\quad\textnormal{and}\quad{y_{n_j}^{**}(t)}\rightarrow{y^{**}(t)}
\end{equation}  
for any $t>0$. Then, since $x_{n_j}\prec{y_{n_j}}$ for all $j\in\mathbb{N}$, this yields that $x\prec{y}$. Thus, since $x,y\in{S_E}$, in view of assumption that $E$ is strictly $K$-monotone we obtain $x=y$ a.e. Furthermore, by \eqref{equ:4:thm:FR*=>DIUKM} and by assumption that $x_{n+1}\prec{x_n}\prec{y_n}$ for any $n\in\mathbb{N}$ we get 
\begin{equation}\label{equ:3:thm:FR*=>DIUKM}
y\prec{y_n}
\end{equation} 
for all $n\in\mathbb{N}$. Next, since compact fully $k$-rotundity implies compact local fully $k$-rotundity on  $E^d$, by assumption that $E^d$ is compactly fully $k$-rotund and strictly $K$-monotone, in view of Theorem \ref{thm:CLFkR*=>ULUKM} we have $E$ is upper locally uniformly $K$-monotone. In consequence, since $y\in{S_E}$, by \eqref{equ:1:thm:FR*=>DIUKM} and \eqref{equ:3:thm:FR*=>DIUKM} we get 
\begin{equation*}
\norm{y_n^*-y}{E}{}\rightarrow{0}.
\end{equation*}
Finally, since $x=y$ a.e., according to \eqref{equ:2:thmFR*=>DIUKM} and by the double extract sequence theorem and by the triangle inequality of the norm in $E$ we conclude
\begin{equation*}
\norm{y_n^*-x_n^*}{E}{}\rightarrow{0},
\end{equation*}
which gives us the end of the proof.
\end{proof}

\begin{theorem}\label{thm:CFkR*=>IUKM}
	Let $E$ be a symmetric space. If $E^d$ is compactly fully $k$-rotund and strictly $K$-monotone, then $E$ is increasing uniformly $K$-monotone.
\end{theorem}

\begin{proof}
	Immediately, by Proposition 4.4 in \cite{CieLewUKM}, we may assume that $(x_n),(y_n)\subset{E}$, $x_{n}\prec{y_n}\prec{y_{n+1}}$ for every $n\in\mathbb{N}$ and 
	\begin{equation}\label{equ:1:thm:CFkR*=>IUKM}
	\norm{x_n}{E}{}\rightarrow{1}\quad\textnormal{and}\quad\norm{y_n}{E}{}\rightarrow{1}.
	\end{equation}
    Then, for any $k$-subsequences $({y_{n,1}}),({y_{n,2}}),\cdots,({y_{n,k}})$ of $(y_n)$ we have $y_n\prec{y_{n,i}}\prec{y_{n+1,i}}$ for all $n\in\mathbb{N}$ and $i\in\{1,2,\cdots,k\}$. So, it is easy to observe that
    \begin{equation*}
    k{y_n}\prec\sum_{i=1}^{k}y_{n,i}^*
    \end{equation*}
    for any $n\in\mathbb{N}$. Therefore, by symmetry of $E$ and by the triangle inequality of the norm in $E$ we obtain
    \begin{equation*}
    k\norm{y_n}{E}{}\leq\norm{\sum_{i=1}^{k}y_{n,i}^*}{E}{}\leq\sum_{i=1}^{k}\norm{y_{n,i}}{E}{}\leq{k\max_{1\leq{i}\leq{k}}}\norm{y_{n,i}}{E}{}
    \end{equation*}
    for all $n\in\mathbb{N}$. Thus, by \eqref{equ:1:thm:CFkR*=>IUKM} we get
    \begin{equation*}
    \norm{\sum_{i=1}^{k}y_{n,i}^*}{E}{}\rightarrow{k}.
    \end{equation*}
    Next, by assumption that $E^d$ is compactly fully $k$-rotund, there exist a subsequence $(y_{n_j})$ of $(y_n)$ and $y\in S_E$ such that
    \begin{equation}\label{equ:2:thm:CFkR*=>IUKM}
    \norm{y_{n_j}^*-y}{E}{}\rightarrow{0}.
    \end{equation}
    Hence, proceeding analogously as in the proof of Theorem \ref{thm:CFkR*=>DUKM} we have $y=y^*$ a.e. and for any $t>0$,
    \begin{equation*}
    y_{n_j}^{**}(t)\rightarrow{y^{**}(t)}.
    \end{equation*}
    Thus, since $x_n\prec{y_n}\prec{y_{n+1}}$ for any $n\in\mathbb{N}$ it is easy to see that 
    \begin{equation}\label{equ:3:thm:CFkR*=>IUKM}
    x_n\prec{y_n}\prec{y}
    \end{equation}
    for any $n\in\mathbb{N}$. Next, by assumption that $E^d$ is compactly fully $k$-rotund and by Proposition \ref{prop:1:CFR*=>OC} we get $E$ is order continuous. So, in view of Lemma 2.5 in \cite{CieKolPan} we obtain $y^*(\infty)=0$. In consequence, since $y\in{S_E}$, by \eqref{equ:1:thm:CFkR*=>IUKM} and \eqref{equ:3:thm:CFkR*=>IUKM} as well as by assumption that $E$ is strictly $K$-monotone, in view of Theorem 1 in \cite{Cies-SKM&KOC} we conclude that
    \begin{equation}\label{equ:4:thm:CFkR*=>IUKM}
   y_{n}^{**}\rightarrow{y^{**}}\quad\textnormal{and}\quad{}x_{n}^{**}\rightarrow{y^{**}}
   \end{equation}
   globally in measure. Now, since $E^d$ is compactly fully $k$-rotund and strictly $K$-monotone, by Theorem \ref{thm:CLFkR*=>ULUKM} we have $E$ is upper locally uniformly $K$-monotone. Hence, since $E$ is order continuous, by \eqref{equ:1:thm:CFkR*=>IUKM} and \eqref{equ:4:thm:CFkR*=>IUKM} as well as by Theorem 3.13 in \cite{Cies-K-mono} we conclude 
   \begin{equation*}
   \norm{y_{n}^*-x_{n}^*}{E}{}\rightarrow{0}, 
   \end{equation*}
   which completes the proof.
\end{proof}

Immediately, by Proposition 3.5 in \cite{CieLewUKM} and Theorem \ref{thm:CFR=>reflexive} we obtain the following relationship between compact fully $k$-rotundity and $K$-order continuity.

\begin{corollary}
Let $E$ be a symmetric space. If $E^d$ is compactly fully $k$-rotund, then the spaces $E$ is $K$-order continuous.
\end{corollary}

\section{Application to approximation problems}

First, for the reader's convenience and the sake of completeness we recall the following characterization of the Kadec-Klee property in symmetric spaces.

\begin{theorem}\label{app:thm:1}
Let $E$ be a symmetric space. If $E$ is order continuous, then the following conditions are equivalent.
\begin{itemize}
	\item[$(i)$] $E$ has the Kadec-Klee property.
	\item[$(ii)$] $E$ is strictly $K$-monotone and has the Kadec-Klee property for global convergence in measure.
	\item[$(iii)$] $E$ is upper locally uniformly $K$-monotone. 
	\item[$(iv)$] $E$ is strictly $K$-monotone and for any $(x_n)\subset{E}$ and $x\in{E}$,
	\begin{equation*}
	x_n^{**}\rightarrow{x^{**}}\quad\textnormal{in measure and}\quad\norm{x_n}{E}{}\rightarrow\norm{x}{E}{}\quad\Rightarrow\quad\norm{x_n^*-x^*}{E}{}\rightarrow{0}.
	\end{equation*}	
\end{itemize}
\end{theorem}

\begin{proof}
Immediately, using the same technique as in the proof of Theorem 2.10 in \cite{ChDSS} and in view of Corollary 1.6 and Proposition 1.7 in \cite{ChDSS} we get $(i)\Leftrightarrow(ii)$. In consequence, by Theorem 3.13 in \cite{Cies-K-mono} we have $(ii)\Leftrightarrow(iii)\Leftrightarrow(iv)$.
\end{proof}

Now, according to Theorem 3 in \cite{HuKoLe} and by Theorem \ref{app:thm:1} we present a correspondence between approximative compactness and $K$-monotonicity properties in symmetric spaces.

\begin{corollary}\label{app:coro:1}
Let $E$ be a symmetric space. The conditions are equivalent.
\begin{itemize}
	\item[$(i)$] $E$ is approximatively compact.
	\item[$(ii)$] $E$ is reflexive and strictly $K$-monotone and has the Kadec-Klee property for global convergence in measure.
	\item[$(iii)$] $E$ is reflexive and upper locally uniformly $K$-monotone. 
	\item[$(iv)$] $E$ is reflexive and strictly $K$-monotone and for any $(x_n)\subset{E}$ and $x\in{E}$,
	\begin{equation*}
	x_n^{**}\rightarrow{x^{**}}\quad\textnormal{in measure and}\quad\norm{x_n}{E}{}\rightarrow\norm{x}{E}{}\quad\Rightarrow\quad\norm{x_n^*-x^*}{E}{}\rightarrow{0}.
	\end{equation*}	
\end{itemize}
\end{corollary}

In the view of the previous result, we present the complete criteria for approximative compactness in the Lorentz space $\Gamma_{p,w}$.

\begin{theorem}\label{app:thm:2}
	Let $1<p<\infty$ and $w$ be a weight function. The statements are equivalent.
	\begin{itemize}
		\item[$(i)$] $\Gamma_{p,w}$ is approximatively compact.
		\item[$(ii)$] $\Gamma_{p,w}$ is reflexive and strictly $K$-monotone.
		\item[$(iii)$] $\Gamma_{p,w}$ is reflexive and $W$ is strictly increasing.
	\end{itemize}
\end{theorem}

\begin{proof}
Immediately, by Theorem 2.10 in \cite{CieKolPluSKM} we have $(ii)\Leftrightarrow(iii)$. Next, by Theorem 4.1 in \cite{CieKolPlu} and by Corollary \ref{app:coro:1} we conclude $(i)\Leftrightarrow(ii)$.
\end{proof}

We investigate reflexivity in the Lorentz spaces $\Gamma_{p,w}$.

\begin{lemma}
Let $1<p,p'<\infty$, $p'p=p'+p$ and let $w\geq{0}$ be a weight function on $[0,\infty)$ such that $\int_{0}^{t}w(s)s^{-p}ds=\infty$ for all $t>0$. The following statements are equivalent.
\begin{itemize}
	\item[$(i)$] The Lorentz space $\Gamma_{p,w}$ is reflexive.
	\item[$(ii)$] 
	$W(\infty)=\int_{0}^{\infty}w(s)ds=\infty$ and $V(\infty)=\int_{0}^{\infty}v(s)ds=\infty,$\\ where  ${v(t)=\frac{t^{p'-1}W(t)W_p(t)}{(W(t)+W_p(t))^{p'+1}}}$ for any $t\in(0,\infty).$
\end{itemize}
\end{lemma}

\begin{proof}
First, by Corollary 4.4 in \cite{BS} we easily observe that the Lorentz space $\Gamma_{p,w}$ is reflexive if and only if $\Gamma_{p,w}$ and its associate space $(\Gamma_{p,w})'$ are order continuous. Next, since $1<p,p'<\infty$ and $p'p=p'+p$ as well as $\int_{0}^{t}w(s)s^{-p}ds=\infty$ for all $t>0$, by Theorem A in \cite{GK} it follows that $(\Gamma_{p,w})'$ coincides with the Lorentz space $\Gamma_{p',v}$ under the assumption that $W(\infty)=\infty$, i.e. we have 
\begin{equation*}\label{equ:approx:GK}
\norm{x}{(\Gamma_{p,w})'}{}\approx\norm{x}{\Gamma_{p',v}}{}\quad\textnormal{for all}\quad{x\in}(\Gamma_{p,w})'.
\end{equation*}
Finally, according to Proposition 1.4 in \cite{KMGam} we conclude $\Gamma_{p,w}$ is reflexive if and only if $W(\infty)=V(\infty)=\infty$.
\end{proof}

Now, we present some examples of the Lorentz spaces which are reflexive and also approximatively compact.

\begin{example}
Let $1<p<\infty$ and $w\geq{0}$ be a weight function such that $W(t)=\int_{0}^{t}w$ satisfies $\Delta_2$ condition and $W(\infty)=\infty$. Define $$v(t)=\left({t}/{W(t)}\right)^{p'}w(t)$$ for any $t>0$, where $p'=p/(p-1)$. Then, by Proposition 0.1 in \cite{KMGam} we get $(\Lambda_{p,w})'=\Gamma_{p',v}$. Next, by Corollary 5.3 in \cite{BS} it follows that $W$ is quasiconcave, and so $t/W(t)$ is increasing on $(0,\infty)$. Hence, taking $t_0>0$ we observe
\begin{align*}
V(\infty)&=\int_{0}^{\infty}\left(\frac{s}{W(s)}\right)^{p'}w(s)ds\geq\int_{t_0}^{\infty}\left(\frac{s}{W(s)}\right)^{p'}w(s)ds\\
&\geq\int_{t_0}^{\infty}\left(\frac{t_0}{W(t_0)}\right)^{p'}w(s)ds=W(\infty)\left(\frac{t_0}{W(t_0)}\right)^{p'}.
\end{align*}
Therefore, since $W(\infty)=\infty$ we have $V(\infty)=\infty$. In consequence, by Proposition 1.4 in \cite{KMGam} it follows that the Lorentz spaces  $\Lambda_{p,w}$ and $\Gamma_{p',v}$ are order continuous. Finally, by Corollary 4.4 and Theorem 2.7 in \cite{BS} we conclude $\Lambda_{p,w}$ and $\Gamma_{p',v}$ are reflexive. Now, if we assume additionally that $W$ is strictly increasing, by definition of $v$ and by Theorem \ref{app:thm:2} we get $\Gamma_{p',v}$ is approximatively compact.
\end{example}

\begin{example}
Consider $1<p<\infty$ and $w\geq{0}$ a weight function such that $W(\infty)=\int_{0}^{1}w(s)s^{-p}ds=\infty$ and $w$ satisfies condition $RB_p$, i.e. there exists $A>0$ such that for all $t>0$ we have $W(t)\leq{A}W_p(t)$. Define $p'=p/(p-1)$ and
$$v(t)=\frac{d}{dt}\left(\int_{t}^{\infty}w(s)s^{-p}ds\right)^{\frac{-1}{p-1}}$$
for any $t>0$. Then, by Corollary 1.9 in \cite{KMGam} it follows that the dual space $(\Gamma_{p,w})^*$ of the Lorentz space $\Gamma_{p,w}$ coincides with $\Lambda_{p',v}$. Next, we notice that
\begin{equation*}
V(\infty)=\lim_{t\rightarrow\infty}\left(\int_{t}^{\infty}w(s)s^{-p}ds\right)^{\frac{-1}{p-1}}=\infty.
\end{equation*} 
Consequently, using the same argumentation as in the previous example we obtain $\Gamma_{p,w}$ and $\Lambda_{p',v}$ are reflexive. Finally, if we suppose additionally that $W$ is strictly increasing, by Theorem \ref{app:thm:2} we get $\Gamma_{p,w}$ is approximatively compact.
\end{example}

The next corollaries follow directly from Corollaries 3.9 and 3.10, Theorem 3.13 in \cite{Cies-JAT} and Theorem \ref{prop:Orlicz:KOC}.

\begin{corollary}
	Let $\psi$ be an Orlicz function and let $\mathcal{A}\subset{L^\psi}$ be a closed subset such that for any $a\in\mathcal{A}$ we have $a^*\in\mathcal{A}$. If $\psi$ satisfies $\Delta_2$ condition and in case when $\alpha=\infty$ we have $\psi$ is $N$-function at zero, then for any $x\in{L^\psi}$ such that $\mathcal{A}\prec{x}$ the set $P_{\mathcal{A}}(x^*)$ is proximinal.
\end{corollary}

\begin{corollary}
	Let $\psi$ be an Orlicz function and in case when $\alpha=\infty$, $a_\psi=0$. If for any $x\in{L^\psi}$ and any closed subset $\mathcal{A}\subset{L^\psi}$ such that $\mathcal{A}\prec{x}$ we have $P_{\mathcal{A}}(x)$ is proximinal, then $L^\psi$ is $K$-order continuous.
\end{corollary}

Let us recall that a point $a\in{E}$ is called a $K$-upper bound of a subset $\mathcal{A}\subset{E}$ if for any $a'\in\mathcal{A}$ we have $a'\prec{a}$. If there exists a $K$-upper bound of a subset $\mathcal{A}\subset{E}$, then the set $\mathcal{A}$ is said to be $K$-bounded above (see \cite{Cies-JAT}).

\begin{corollary}
Let $\psi$ be an Orlicz function and in case when $\alpha=\infty$, $a_\psi=0$. The conditions are equivalent.
\begin{itemize}
	\item[$(i)$]  For any $x\in{L^\psi}$ and $\mathcal{A}\subset{L^\psi}$ a closed $K$-bounded above subset such that $x\prec\mathcal{A}$, $a^*\in\mathcal{A}$ for any $a\in\mathcal{A}$ we have $P_\mathcal{A}(x^*)$ is proximinal.
	\item[$(ii)$] $\psi$ satisfies $\Delta_2$ condition and if $\alpha=\infty$, then $\psi$ is $N$-function at zero.
\end{itemize}
\end{corollary}

\subsection*{Acknowledgments}
\begin{flushleft}
  This research is supported by the grant 2017/01/X/ST1/01036 from National Science Centre, Poland.	
\end{flushleft}

$\begin{array}{lr}
\textnormal{\small Maciej Ciesielski}\\
\textnormal{\small Institute of Mathematics}\\
\textnormal{\small Pozna\'{n} University of Technology} \\
\textnormal{\small Piotrowo 3A, 60-965 Pozna\'{n}, Poland}\\
\textnormal{\small email: maciej.ciesielski@put.poznan.pl} 
\end{array}$


\begin{thebibliography}{99}
	
	\bibitem{BS} C. Bennett and R. Sharpley, \textit{Interpolation of operators}, Pure and Applied Mathematics, 129. Academic Press, Inc., Boston, MA, 1988.
		
	\bibitem{Cal} A. P. Calder\'{o}n, \textit{Intermediate spaces and interpolation, the complex method}, Studia Math. \textbf{24} (1964),
	113-190.
	
	\bibitem{CerHudzKamMas} J. Cerd\`{a}, H. Hudzik, A. Kami\'nska and M. Masty\l o, \textit{Geometric properties of symmetric spaces with applications to Orlicz-Lorentz spaces}, Positivity \textbf{2} (1998), no. 4, 311-337.
	
	\bibitem{ChDSS} V. I. Chilin, P. G. Dodds, A. A. Sedaev, and F. A. Sukochev, \textit{Characterizations of Kadec-Klee properties in symmetric spaces of measurable functions}, Trans. Amer. Math. Soc. \textbf{348} (1996), no. 12, 4895-4918.
	
	\bibitem{Cies-K-mono} M. Ciesielski, \textit{Lower and upper local uniform $K$-monotonicity in symmetric spaces}, Banach J. Math. Anal., advance publication, 19 December 2017. doi:10.1215/17358787-2017-0047. https://projecteuclid.org/euclid.bjma/1513674116.
	
	\bibitem{Cies-SKM&KOC} M. Ciesielski, \textit{Strict $K$-monotonicity and $K$-order continuity in symmetric spaces}, Positivity (2017) https://doi.org/10.1007/s11117-017-0540-7
	
	\bibitem{Cies-JAT} M. Ciesielski, \textit{Hardy-Littlewood-P\'olya relation in the best dominated approximation in symmetric spaces}, J. Approx. Theory \textbf{213} (2017), 78-91.
	
	\bibitem{CieKolPan} M. Ciesielski, P. Kolwicz and A. Panfil, \textit{Local monotonicity structure of symmetric spaces with applications}, J. Math. Anal. Appl. \textbf{409} (2014), no. 2, 649-662.
	
	\bibitem{CieKolPlu} M. Ciesielski, P. Kolwicz and R. P\l uciennik, \textit{Local approach to Kadec-Klee properties in symmetric function spaces}, J.Math. Anal. Appl. \textbf{426} (2015), no. 2, 700-726.
	
	\bibitem{CieKolPluSKM} M. Ciesielski, P. Kolwicz and R. P\l uciennik, \textit{A note on strict $K$-monotonicity of some symmetric function spaces}, Comm. Math. \textbf{53} (2013), no. 2, 311-322.  

    \bibitem{CieLewUKM} M. Ciesielski and G. Lewicki, \textit{Uniform $K$-monotonicity and $K$-order continuity in symmetric spaces with application to approximation theory}, J. Math. Anal. Appl. \textbf{456} (2017), no. 2, 705-730.
	
	\bibitem{CzeKam} M. M. Czerwi\'nska, A. Kami\'nska, \textit{Complex rotundities and midpoint local uniform rotundity in symmetric spaces of measurable operators}, Studia Math. \textbf{201} (2010), no. 3, 253-285.
	
	\bibitem{CHK} Y. Cui, H. Hudzik and W. Kowalewski, \textit{On fully rotundity properties and approximative compactness in some Banach sequence spaces}, Indian J. Pure Appl. Math. \textbf{34} (2003), no. 1, 17–30.
	
	\bibitem{DSS}  P.G. Dodds, E.M. Semenov and F.A. Sukochev, \textit{The Banach-Saks property in rearrangement invariant spaces}, Studia Math. \textbf{162} (2004), no. 3, 263-294.

	\bibitem{FaGl} K. Fan and I. Glicksberg, \textit{Some geometric properties of the spheres in a normed linear space}, Duke Math. J. \textbf{25} (1958), 553-568.
	
	\bibitem{GK} A. Gogatishvili and R. Kerman, \textit{The rearrangement-invariant space $\Gamma_{p,\phi}$}, (English summary) Positivity \textbf{18} (2014), no. 2, 319-345.
	\bibitem{hkm-geo-prop} H. Hudzik, A. Kami\'nska and M. Masty\l o, \textit{On geometric properties of Orlicz-Lorentz spaces}, Canad. Math. Bull. \textbf{40} (1997), no. 3, 316-329.
	
	\bibitem{HuKoLe}  H. Hudzik, W. Kowalewski, and G. Lewicki, \textit{Approximate compactness and full rotundity in Musielak-Orlicz spaces and Lorentz-Orlicz spaces}, Z. Anal. Anwend. \textbf{25} (2006), no. 2, 163-192.
	
	\bibitem{KMGam} A. Kami\'{n}ska and L. Maligranda, \textit{On Lorentz spaces 
	}$\Gamma _{p,w}$, Israel J. Math. 140 (2004), 285-318.
	
	\bibitem{KamMal} A. Kami\'nska and L. Maligranda, \textit{Order convexity and concavity of Lorentz spaces $\Lambda_{p,w}$, $0<p<\infty$}, Studia Math. \textbf{160} (2004), no. 3, 267-286.
	
	\bibitem{KraRut} M. A. Krasnoselski\u{\i} and Ja. B. Ruticki\u{\i}, \textit{Convex functions and Orlicz spaces. Translated from the first Russian edition by Leo F. Boron}, P. Noordhoff Ltd., Groningen 1961
	
	\bibitem{KPS} S. G. Krein, Yu. I. Petunin and E. M. Semenov, \textit{Interpolation of linear operators}, Translated from the Russian by J. Sz\H{u}cs. Translations of Mathematical Monographs, 54. American Mathematical Society, Providence, R.I., 1982.
	
	\bibitem{LinTza} J. Lindenstrauss and L. Tzafriri, \textit{Classical Banach spaces. II. Function spaces}, Ergebnisse der Mathematik und ihrer Grenzgebiete [Results in Mathematics and Related Areas], 97. Springer-Verlag, Berlin-New York, 1979.
	
	\bibitem{Loren} G. G. Lorentz, \textit{On the theory of spaces }$\Lambda $, Pacific J. Math. \textbf{1} (1951), 411-429.
	
	\bibitem{Meggin} R.E. Megginson, \textit{An introduction to Banach space theory}, (English summary) Graduate Texts in Mathematics, 183. Springer-Verlag, New York, 1998.
	
	\bibitem{Royd} H. L. Royden, \textit{Real analysis}, Third edition, Macmillan Publishing Company, New York, 1988.
	
	\bibitem{Haaker} A. Sparr, \textit{On the conjugate space of the Lorentz space $L(\phi,q)$}, (English summary) Interpolation theory and applications, 313-336, 
	Contemp. Math., 445, Amer. Math. Soc., Providence, RI, 2007.
	
\end{thebibliography}
\end{document}